\documentclass[aop,MSNbibl,citesort,dvips]{arximspdf}
\usepackage{mathrsfs}
\usepackage{euscript}
\usepackage{graphicx}

\doi{10.1214/10-AOP619}
\volume{40}
\issue{1}
\pubyear{2012}
\firstpage{401}
\lastpage{435}

\makeatletter
\renewcommand{\epsilon}{\varepsilon}
\newcommand{\dotsc}{\ldots}
\newcommand{\dotsb}{\cdots}
\newcommand{\fracb}[2]{{(#1)/#2}}
\newtheorem{theorem}{Theorem}[section]
\newtheorem{lemma}[theorem]{Lemma}
\newtheorem{proposition}[theorem]{Proposition}
\newtheorem{maintheorem}{Theorem}
\newproclaim{definition}[theorem]{Definition}
\newproclaim{example}[theorem]{Example}
\newproclaim{notation}[theorem]{Notation}
\newproclaim{remark}[theorem]{Remark}

\newcommand{\R}{\mathbb R}
\newcommand{\C}{\mathbb C}
\newcommand{\Z}{\mathbb Z}
\newcommand{\N}{\mathbb N}
\newcommand{\tr}{\mathrm{tr}}
\newcommand{\Hom}{\operatorname{Hom}}

\newcommand{\ms}[1]{\mathscr{#1}}
\newcommand{\mc}[1]{\mathcal{#1}}
\newcommand{\mb}[1]{\mathbb{#1}}
\newcommand{\eu}[1]{\EuScript{#1}}

\makeatother

\begin{document}
\begin{frontmatter}

\title{De Finetti theorems for easy quantum groups}
\runtitle{De Finetti theorems}

\begin{aug}
\author[A]{\fnms{Teodor} \snm{Banica}\thanksref{a1}\ead[label=e1]{teodor.banica@u-cergy.fr}},
\author[B]{\fnms{Stephen} \snm{Curran}\corref{}\ead[label=e2]{curransr@math.ucla.edu}}
\and
\author[C]{\fnms{Roland} \snm{Speicher}\thanksref{a2}\ead[label=e3]{speicher@math.uni-sb.de}}
\runauthor{T. Banica, S. Curran and R. Speicher}
\affiliation{Cergy-Pontoise University, UCLA and Queen's University}
\address[A]{T. Banica\\
Department of Mathematics\\
Cergy-Pontoise University\\
95000 Cergy-Pontoise\\
France\\
\printead{e1}} 
\address[B]{S. Curran\\
Department of Mathematics\\
UCLA\\
Los Angeles, California 90095\\
USA\\
\printead{e2}}
\address[C]{R. Speicher\\
Department of Mathematics\\
 \quad  and Statistics\\
Queen's University\\
Jeffery Hall\\
Kingston, Ontario K7L 3N6\\
Canada\\
and\\
Saarland University\\
FR 6.1 - Mathematik\\
Campus E 2.4, 66123 Saarbrucken\\
Germany\\
\printead{e3}}
\end{aug}
\thankstext{a1}{Supported by ANR grants ``Galoisint'' and ``Granma.''}
\thankstext{a2}{Supported by a Discovery grant from NSERC.}
\received{\smonth{4} \syear{2010}}

%
\begin{abstract}
We study sequences of noncommutative random variables which are
invariant under ``quantum transformations'' coming from an orthogonal
quantum group satisfying the ``easiness'' condition axiomatized in our
previous paper. For 10 easy quantum groups, we obtain de Finetti type
theorems characterizing the joint distribution of any infinite quantum
invariant sequence. In particular, we give a new and unified proof of
the classical results of de Finetti and Freedman for the easy groups
$S_n, O_n$, which is based on the combinatorial theory of cumulants. We
also recover the free de Finetti theorem of K\"{o}stler and Speicher,
and the characterization of operator-valued free semicircular families
due to Curran. We consider also finite sequences, and prove an
approximation result in the spirit of Diaconis and Freedman.
\end{abstract}

%

\end{frontmatter}

\section*{Introduction}

In the study of probabilistic symmetries, the classical\break groups $S_n$
and $O_n$ play central roles. De Finetti's fundamental theorem states
that an infinite sequence of random variables whose joint distribution
is invariant under finite permutations must be conditionally
independent and identically distributed. In~\cite{fre1}, Freedman
considered sequences of real-valued random variables whose joint
distribution is invariant under orthogonal transformations, and proved
that any infinite sequence with this property must form a conditionally
independent Gaussian family with mean zero and common variance.
Although these results fail for finite sequences, approximation results
may still be obtained (see~\cite{df1,df2}). For a thorough treatment
of probabilistic symmetries, the reader is referred to the recent text
of Kallenberg~\cite{kal}.

The free analogues $S_n^+$ and $O_n^+$ of the permutation and
orthogonal groups were constructed by Wang in~\cite{wang1,wang2}.
These are compact quantum groups in the sense of Woronowicz \cite
{wor1}. In~\cite{ksp}, K\"{o}stler and Speicher discovered that de
Finetti's theorem has a natural free analogue: an infinite sequence of
noncommutative random variables has a joint distribution which is
invariant under ``quantum permutations'' coming from $S_n^+$ if and
only if the variables are freely independent and identically
distributed with amalgamation, that is, with respect to a conditional
expectation. This was further studied in~\cite{cur1}, where this
result was extended to more general sequences and an approximation
result was given for finite sequences. The free analogue of Freedman's
result was obtained in~\cite{cur2}, where it was shown that an
infinite sequence of self-adjoint noncommutative random variables has
a~joint distribution which is invariant under ``quantum orthogonal
transformations'' if and only if the variables form an operator-valued
free semicircular family with mean zero and common variance.

In this paper, we present a unified approach to de Finetti theorems by
using the ``easiness'' formalism from~\cite{bsp}. Stated roughly, a
quantum group $S_n \subset G \subset O_n^+$ is called \textit{easy} if
its tensor category is spanned by certain partitions coming from the
tensor category of $S_n$. This might look, of course, to be a quite
technical condition. However, we feel that this provides a good
framework for understanding certain probabilistic and representation
theory aspects of orthogonal quantum groups. There are 14 natural
examples of easy quantum groups, listed as follows:
\begin{longlist}[(3)]
\item[(1)] Groups: $O_n, S_n, H_n, B_n, S_n', B_n'$.
\item[(2)] Free versions: $O_n^+,S_n^+,H_n^+,B_n^+,S_n'^+,B_n'^+$.
\item[(3)] Half-liberations: $O_n^*, H_n^*$.
\end{longlist}

Except for $H_n^*$, which was found in~\cite{bcs1}, these are all
described in~\cite{bsp}. The four ``primed'' versions above are rather
trivial modifications of their ``unprimed'' versions, corresponding to
taking a product with a copy of $\Z_2$. We will focus then on the
remaining ten examples in this paper, from which similar results for the
``primed'' versions may be easily deduced.

As explained in~\cite{bsp,bcs1}, our motivating belief is that ``any
result which holds for $S_n, O_n$ should have a suitable extension to
all easy quantum groups.'' This is, of course, a quite vague statement,
whose target is formed by several results at the borderline of
representation theory and probability. This paper represents the first
application of this philosophy.

If $G$ is an easy quantum group, there is a natural notion of
$G$-invariance for a sequence of noncommutative random variables, which
agrees with the usual definition when $G$ is a classical group. Our
main result is the following de Finetti type theorem, which
characterizes the joint distributions of infinite $G$-invariant
sequences for the 10 natural easy quantum groups discussed above.
\begin{maintheorem}\label{definetti}
Let $(x_i)_{i \in\N}$ be a sequence of self-adjoint random variables
in a W$^*$-probability space $(M,\varphi)$, and suppose that the\vadjust{\goodbreak}
sequence is $G$-invariant, where $G$ is one of
$O,S,H,B,O^*,H^*,O^+,S^+,H^+,B^+$. Assume that $M$ is generated as a
von Neumann algebra by $\{x_i\dvtx  i \in\N\}$. Then there is a
W$^*$-subalgebra $1 \subset\mc B \subset M$ and a $\varphi
$-preserving conditional expectation $E\dvtx M \to\mc B$ such that the
following hold:
\begin{enumerate}[(3)]
\item[(1)] Free case:
\begin{longlist}[(b)]
\item[(a)] If $G = S^+$, then $(x_i)_{i \in\N}$ are freely independent and
identically distributed with amalgamation over $\mc B$.
\item[(b)] If $G = H^+$, then $(x_i)_{i \in\N}$ are freely independent,
and have even and identical distributions, with amalgamation over $\mc B$.
\item[(c)] If $G = O^+$, then $(x_i)_{i \in\N}$ form a $\mc B$-valued free
semicircular family with mean zero and common variance.
\item[(d)] If $G = B^+$, then $(x_i)_{i \in\N}$ form a $\mc B$-valued free
semicircular family with common mean and variance.
\end{longlist}
\item[(2)] Half-liberated case: Suppose that $x_ix_jx_k = x_kx_jx_i$ for any
$i,j,k \in\N$.
\begin{longlist}[(b)]
\item[(a)] If $G = H^*$, then $(x_i)_{i \in\N}$ are conditionally
half-independent and identically distributed given $\mc B$.
\item[(b)] If $G = O^*$, then $(x_i)_{i \in\N}$ are conditionally
half-independent, and have symmetrized Rayleigh distributions with
common variance, given~$\mc B$.
\end{longlist}
\item[(3)] Classical case: Suppose that $(x_i)_{i \in\N}$ commute.
\begin{longlist}[(b)]
\item[(a)] If $G = S$, then $(x_i)_{i \in\N}$ are conditionally
independent and identically distributed given $\mc B$.
\item[(b)] If $G = H$, then $(x_i)_{i \in\N}$ are conditionally
independent, and have even and identical distributions, given $\mc B$.
\item[(c)] If $G = O$, then $(x_i)_{i \in\N}$ are conditionally
independent, and have Gaussian distributions with mean zero and common
variance, given $\mc B$.
\item[(d)] If $G = B$, then $(x_i)_{i \in\N}$ are conditionally
independent, and have Gaussian distributions with common mean and
variance, given $\mc B$.
\end{longlist}
\end{enumerate}
\end{maintheorem}

The notion of half-independence, appearing in (2) above, will be
introduced in Section~\ref{sec:halfindependence}. The basic example of
a half-independent family of noncommutative random variables is
$(x_i)_{i \in I}$,
\[
x_i =
\pmatrix{\displaystyle 0 & \xi_i\cr\displaystyle  \overline{\xi_i} & 0
}
,
\]
where $(\xi_i)_{i \in I}$ are independent, complex-valued random
variables and\break $\mb E[\xi_i^n\overline{\xi_i}^m] = 0$ unless $n = m$
(see Example~\ref{halfexample} and Proposition \ref
{halfindependentform}). Note that in particular, if $(\xi_i)_{i \in\N
}$ are independent and identically distributed complex Gaussian random
variables, then $x_i$ has a symmetrized Rayleigh distribution $(\xi
_i\overline{\xi_i})^{1/2}$ and we obtain the joint distribution in
(2) corresponding to the half-liberated\vadjust{\goodbreak} orthogonal group $O^*$. Since
the complex Gaussian distribution is known to be characterized by
unitary invariance, this appears to be closely related to the
connection between $U_n$ and $O_n^*$ observed in~\cite{bv,bcs1}.

Let us briefly outline the proof of Theorem~\ref{definetti}, to be
presented in Section~\ref{sec:definetti}. We define von Neumann
subalgebras $\mc B_n \subset M$ consisting of ``functions'' of the
variables $(x_i)_{i \in\N}$ which are invariant under ``quantum
transformations'' of $x_1,\dotsc,x_n$ coming from the quantum group
$G_n$. The \textit{$G$-invariant} subalgebra~$\mc B$ is defined as the
intersection of the nested sequence $\mc B_n$ (note that if $G = S$,
then $\mc B$ corresponds to the classical \textit{exchangeable
subalgebra}). There are natural conditional expectations onto $\mc B_n$
given by ``averaging'' over $G_n$, that is, integrating with respect to
the \textit{Haar state} on the compact quantum group $G_n$. By using
an explicit formula for the Haar states on easy quantum groups from
\cite{bsp}, and a noncommutative reversed martingale convergence
argument, we obtain a simple combinatorial formula for computing joint
moments with respect to the conditional expectation onto the
$G$-invariant subalgebra. What emerges from these computations is a
\textit{moment-cumulant formula}, and Theorem~\ref{definetti} follows
from the characterizations of these joint distributions by the
structure of their cumulants. Note that, in particular, we obtain a new
proof of de Finetti's classical result for $S_n$ which is based on
cumulants. This method also allows us to give certain approximation
results for finite sequences, which will be explained in Section \ref
{sec:finite}.

The paper is organized as follows. Section~\ref{sec:background}
contains preliminaries. Here we collect the basic notions from the
combinatorial theory of classical and free probability. We also recall
some basic notions and results from~\cite{bsp} about the class of
``easy'' quantum groups. In Section~\ref{sec:halfindependence}, we
introduce half-independence and develop its basic combinatorial theory.
In Section~\ref{sec:weingarten}, we recall the Weingarten formula from
\cite{bsp} for computing integrals on easy quantum groups, and give a
new estimate on the asymptotic behavior of these integrals. This will
be essential to the proofs of our main results, and we believe that
this estimate will also find applications to other problems involving
easy quantum groups. In Section~\ref{sec:finite}, we define quantum
invariance for finite sequences, prove a converse to Theorem \ref
{definetti}, and give approximate de Finetti type results. Section \ref
{sec:definetti} contains the proof of Theorem~\ref{definetti}, and
a~discussion of the situation for unbounded random variables in the
classical and half-liberated cases. Section~\ref{sec:conclusion}
contains concluding remarks.

\section{Background and notation}\label{sec:background}

\subsection*{Noncommutative probability} We begin by recalling
the basic notions of noncommutative probability spaces and
distributions of random variables. For further details, see the texts
\cite{vdn,ns}.
\begin{definition}
{\baselineskip=0pt\begin{longlist}[(2)]
\item[(1)] A \textit{noncommutative probability space} is a pair $(\mc
A,\varphi)$, where $\mc A$ is a~unital algebra over $\C$, and
$\varphi\dvtx \mc A \to\C$\vadjust{\goodbreak} is a linear functional such that $\varphi(1)
= 1$. Elements in a noncommutative probability space $(\mc A,\varphi)$
will be called \textit{noncommutative random variables}, or simply
\textit{random variables}.
\item[(2)] A W$^*$-probability space $(M,\varphi)$ is a von Neumann
algebra $M$ together with a faithful normal state $\varphi$. We will
not assume that $\varphi$ is a trace.
\end{longlist}}
\end{definition}

\begin{example}
Let $(\Omega,\Sigma,\mu)$ be a (classical) probability space.
\begin{longlist}[(2)]
\item[(1)] The pair $(L^\infty(\mu),\mb E)$ is a W$^*$-probability space,
where $L^\infty(\mu)$ is the algebra of bounded $\Sigma$-measurable
random variables, and $\mb E$ is the expectation functional $\mb E(f) =
\int f \,d\mu$.
\item[(2)] Let
\[
L(\mu) = \bigcap_{1 \leq p < \infty} L^p(\mu)
\]
be the algebra of random variables with finite moments of all orders.
Then $(L(\mu),\mb E)$ is a noncommutative probability space.
\end{longlist}
\end{example}

The joint distribution of a sequence $(X_1,\dotsc,X_n)$ of (classical)
random variables can be defined as the linear functional on $C_b(\R
^n)$ determined by
\[
f \mapsto\mb E[f(X_1,\dotsc,X_n)].
\]
In the noncommutative context, it is generally not possible to make
sense of $f(x_1,\dotsc,x_n)$ for $f \in C_b(\R^n)$ if the random
variables $x_1,\dotsc,x_n$ do not commute. Instead, we work with an
algebra of noncommutative polynomials.

\begin{notation}
Let $I$ be a nonempty set. We let $\ms P_I$ denote the algebra $\C
\langle t_i\dvtx  i \in I \rangle$ of noncommutative polynomials, with
generators indexed by the set~$I$. Note that $\ms P_I$ is spanned by
$1$ and monomials of the form $t_{i_1}\dotsb t_{i_k}$, for $k \in\N$
and $i_1,\dotsc,i_k \in I$. If $I = \{1,\dotsc,n\}$, we set $\ms P_n
= \ms P_I$, and if $I = \N$ we denote $\ms P_\infty= \ms P_I$.
\end{notation}

Given a family $(x_i)_{i \in I}$ of noncommutative random variables in
a noncommutative probability space $(\mc A,\varphi)$, there is a
unique unital homomorphism $\mathrm{ev}_x\dvtx \ms P_I \to\mc A$ which
sends $t_i$ to $x_i$ for each $i \in I$. We also denote this map by $p
\mapsto p(x)$.

\begin{definition}
Let $(x_i)_{i \in I}$ be a family of random variables in the
noncommutative probability space $(\mc A,\varphi)$. The \textit{joint
distribution} of $(x_i)_{i \in I}$ is the linear functional $\varphi
_x\dvtx \ms P_I \to\C$ defined by
\[
\varphi_x(p) = \varphi(p(x)).
\]
\end{definition}

Note that the joint distribution of $(x_i)_{i \in I}$ is determined by
the collection of \textit{joint moments}
\[
\varphi_x(t_{i_1}\dotsb t_{i_k}) = \varphi(x_{i_1}\dotsb x_{i_k})
\]
for $k \in\N$ and $i_1,\dotsc,i_k \in I$.\vadjust{\goodbreak}

\begin{remark}
In the classical de Finetti's theorem, the independence which occurs is
only after conditioning. Likewise the free de Finetti's theorem is a
statement about freeness with amalgamation. Both of these concepts may
be expressed in terms of operator-valued probability spaces, which we
now recall.
\end{remark}

\begin{definition}
An \textit{operator-valued probability} space $(\mc A, E\dvtx \mc A \to\mc
B)$ consists of a unital algebra $\mc A$, a subalgebra $1 \in\mc B
\subset\mc A$, and a conditional expectation $E\dvtx \mc A \to\mc B$, that
is, $E$ is a linear map such that $E[1] = 1$ and
\[
E[b_1ab_2] = b_1E[a]b_2
\]
for all $b_1,b_2 \in\mc B$ and $a \in\mc A$.
\end{definition}

\begin{example}
Let $(\Omega,\Sigma,\mu)$ be a probability space, and let $\mc F
\subset\Sigma$ be a $\sigma$-subalgebra. Let $\mc A = L^\infty(\mu
)$, and let $\mc B = L^\infty(\mu|_{\mc F})$ be the subalgebra of
bounded, $\mc F$-measurable functions on $\Omega$. Then $(\mc A, \mb
E[\cdot|\mc F])$ is an operator-valued probability space.
\end{example}

To define the joint distribution of a family $(x_i)_{i \in I}$ in an
operator-valued probability space $(\mc A,E\dvtx \mc A \to\mc B)$, we will
use the algebra $\mc B\langle t_i\dvtx  i \in I \rangle$ of noncommutative
polynomials with coefficients in $\mc B$. This algebra is spanned by
monomials of the form $b_0t_{i_1}\dotsb t_{i_k}b_k$, for $k \in\N$,
$b_0,\dotsc,b_k \in\mc B$ and \mbox{$i_1,\dotsc,i_k \in I$}. There is a
unique homomorphism from $\mc B \langle t_i\dvtx i \in I \rangle$ into $\mc
A$ which acts as the identity on $\mc B$ and sends $t_i$ to $x_i$,
which we denote by $p \mapsto p(x)$.

\begin{definition}
Let $(\mc A,E\dvtx \mc A \to\mc B)$ be an operator-valued probability
space, and let $(x_i)_{i \in I}$ be a family in $\mc A$. The \textit
{$B$-valued joint distribution} of the family $(x_i)_{i \in I}$ is the
linear map $E_x\dvtx \mc B \langle t_i\dvtx  i \in I \rangle\to\mc B$ defined by
\[
E_x(p) = E[p(x)].
\]
\end{definition}

Observe that the joint distribution is determined by the \textit{$\mc
B$-valued joint moments}
\[
E_x[b_0t_{i_1}\dotsb t_{i_k}b_k] = E[b_0x_{i_1}\dotsb x_{i_k}b_k]
\]
for $b_0,\dotsc,b_k \in\mc B$ and $i_1,\dotsc,i_k \in I$. Observe
that if $\mc B$ commutes with the variables $(x_i)_{i \in I}$, then
\[
E[b_0x_{i_1}\dotsb x_{i_k}b_k] = b_0\dotsb b_k E[x_{i_1}\dotsb x_{i_k}],
\]
so that the $\mc B$-valued joint distribution is determined simply by
the collection of moments $E[x_{i_1}\dotsb x_{i_k}]$ for $i_1,\dotsc
,i_k \in I$.

\begin{definition}Let $(x_i)_{i \in I}$ be a family in the
operator-valued probability space $(\mc A,E\dvtx \mc A \to\mc B)$.
\begin{longlist}[(2)]
\item[(1)] If the algebra generated by $\mc B$ and $\{x_i\dvtx i \in I\}$ is
commutative, then the variables are called \textit{conditionally
independent given $B$} if
\[
E[p_1(x_{i_1})\dotsb p_k(x_{i_k})] = E[p_1(x_{i_1})]\dotsb
E[p_k(x_{i_k})],
\]
whenever $i_1,\dotsc,i_k$ are distinct and $p_1,\dotsc,p_k$ are
polynomials in $\mc B \langle t \rangle$.

\item[(2)] The variables $(x_i)_{i \in I}$ are called \textit{free with
amalgamation over $\mc B$}, or \textit{free with respect to $E$}, if
\[
E[p_1(x_{i_1})\dotsb p_k(x_{i_k})] = 0,
\]
whenever $i_1,\dotsc,i_k \in I$ are such that $i_l \neq i_{l+1}$ for
$1 \leq l < k$, and $p_1,\dotsc,p_k \in\mc B \langle t \rangle$ are
such that $E[p_l(x_{i_l})] = 0$ for $1 \leq l \leq k$.
\end{longlist}
\end{definition}

\begin{remark}
Voiculescu first defined freeness with amalgamation, and developed its
basic theory in~\cite{voi}. Conditional independence and freeness with
amalgamation also have rich combinatorial theories, which we now
recall. In the free case this is due to Speicher~\cite{sp2}; see also
~\cite{ns}.
\end{remark}

\begin{definition}
{\baselineskip=0pt\begin{longlist}[(5)]
\item[(1)] A \textit{partition} $\pi$ of a set $S$ is a collection of
disjoint, nonempty sets $V_1,\dotsc,V_r$ such that $V_1 \cup\dotsb
\cup V_r = S$. $V_1,\dotsc,V_r$ are called the \textit{blocks} of
$\pi$, and we set $|\pi| = r$. The collection of partitions of $S$
will be denoted $P(S)$, or in the case that $S =\{1,\dotsc,k\}$ by $P(k)$.
\item[(2)] Given $\pi,\sigma\in P(S)$, we say that $\pi\leq\sigma$ if
each block of $\pi$ is contained in a block of $\sigma$. There is a
least element of $P(S)$ which is larger than both~$\pi$ and $\sigma$,
which we denote by $\pi\vee\sigma$.
\item[(3)] If $S$ is ordered, we say that $\pi\in P(S)$ is \textit
{noncrossing} if whenever $V,W$ are blocks of $\pi$ and $s_1 < t_1 <
s_2 < t_2$ are such that $s_1,s_2 \in V$ and $t_1,t_2 \in W$, then $V =
W$. The set of noncrossing partitions of $S$ is denoted by $\mathit{NC}(S)$, or
by $\mathit{NC}(k)$ in the case that $S = \{1,\dotsc,k\}$.

\item[(4)] The noncrossing partitions can also be defined recursively, a
partition $\pi\in P(S)$ is noncrossing if and only if it has a block
$V$ which is an interval, such that $\pi\setminus V$ is a noncrossing
partition of $S \setminus V$.
\item[(5)] Given $i_1,\dotsc,i_k$ in some index set $I$, we denote by
$\ker\mathbf i$ the element of $P(k)$ whose blocks are the equivalence
classes of the relation
\[
s \sim t  \quad \Leftrightarrow \quad  i_s= i_t.
\]
Note that if $\pi\in P(k)$, then $\pi\leq\ker\mathbf i$ is
equivalent to the condition that whenever $s$ and $t$ are in the same
block of $\pi$, $i_s$ must equal $i_t$.
\end{longlist}}
\end{definition}

\begin{definition}
Let $(\mc A, E\dvtx \mc A \to\mc B)$ be an operator-valued probability space.
\begin{longlist}[(2)]
\item[(1)] A \textit{$\mc B$-functional} is a $n$-linear map $\rho\dvtx \mc A^n
\to\mc B$ such that
\[
\rho(b_0a_1b_1,a_2b_2,\dotsc,a_nb_n) = b_0\rho(a_1,b_1a_2,\dotsc
,b_{n-1}a_n)b_n
\]
for all $b_0,\dotsc,b_n \in\mc B$ and $a_1,\dotsc,a_n$.
Equivalently, $\rho$ is a linear map from $\mc A^{\otimes_{\mc B} n}$
to $\mc B$, where the tensor product is taken with respect to the
natural $\mc B - \mc B$-bimodule structure on $\mc A$.
\item[(2)] Suppose that $\mc B$ is commutative. For $k \in\N$ let $\rho
^{(k)}$ be a $\mc B$-functional. Given $\pi\in P(n)$, we define a $\mc
B$-functional $\rho^{(\pi)}\dvtx \mc A^n \to\mc B$ by the formula
\[
\rho^{(\pi)}[a_1,\dotsc,a_n] = \prod_{V \in\pi} \rho
(V)[a_1,\dotsc,a_n],
\]
where if $V = (i_1 < \dotsb< i_s)$ is a block of $\pi$ then
\[
\rho(V)[a_1,\dotsc,a_n] = \rho_s(a_{i_1},\dotsc,a_{i_s}).
\]
\end{longlist}
\end{definition}

If $\mc B$ is noncommutative, there is no natural order in which to
compute the product appearing in the above formula for $\rho^{(\pi
)}$. However, the nesting property of noncrossing partitions allows for
a natural definition of $\rho^{(\pi)}$ for $\pi\in \mathit{NC}(n)$, which we
now recall from~\cite{sp2}.

\begin{definition}
Let $(\mc A,E\dvtx \mc A \to\mc B)$ be an operator-valued probability
space, and for $k \in\N$ let $\rho^{(k)}\dvtx \mc A^k \to\mc B$ be a
$\mc B$-functional. Given $\pi\in \mathit{NC}(n)$, define a $\mc B$-functional
$\rho^{(n)}\dvtx \mc A^n \to\mc B$ recursively as follows:
\begin{longlist}[(2)]
\item[(1)] If $\pi= 1_n$ is the partition containing only one block, define
$\rho^{(\pi)} = \rho^{(n)}$.
\item[(2)] Otherwise, let $V = \{l+1,\dotsc,l+s\}$ be an interval of $\pi$
and define
\[
\rho^{(\pi)}[a_1,\dotsc,a_n] = \rho^{(\pi\setminus V)}\bigl[a_1,\dotsc
,a_l\cdot\rho^{(s)}(a_{l+1},\dotsc,a_{l+s}),a_{l+s+1},\dotsc,a_n\bigr]
\]
for $a_1,\dotsc,a_n \in\mc A$.
\end{longlist}
\end{definition}

\begin{example}
Let
\begin{eqnarray*}
\pi= \{\{1,8,9,10\},\{2,7\},\{3,4,5\}, \{6\}\} \in \mathit{NC}(10),\\[-12pt]
\end{eqnarray*}

\includegraphics{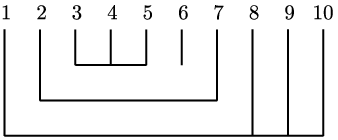}

\vspace*{6pt}
\noindent then the corresponding $\rho^{(\pi)}$ is given by
\[
\rho^{(\pi)}[a_1,\dotsc,a_{10}] = \rho^{(4)}\bigl(a_1\cdot\rho
^{(2)}\bigl(a_2\cdot\rho^{(3)}(a_3,a_4,a_5),\rho^{(1)}(a_6)\cdot
a_7\bigr),a_8,a_9,a_{10}\bigr).
\]
\end{example}

\begin{definition}
Let $(\mc A, E\dvtx \mc A \to\mc B)$ be an operator-valued probability
space, and let $(x_i)_{i \in I}$ be a family of random variables in
$\mc A$.
\begin{longlist}[(2)]
\item[(1)] The \textit{operator-valued classical cumulants} $c_E^{(k)}\dvtx \mc
A^k \to\mc B$ are the $\mc B$-func\-tionals defined by the \textit
{classical moment-cumulant formula}
\[
E[a_1 \dotsb a_n] = \sum_{\pi\in P(n)} c_E^{(\pi)}[a_1,\dotsc,a_n].\vadjust{\goodbreak}
\]
Note that the right-hand side of the equation is equal to
$c_E^{(n)}[a_1,\dotsc,a_n]$ plus lower order terms, and hence
$c_E^{(n)}$ can be solved for recursively.
\item[(2)] The \textit{operator-valued free cumulants} $\kappa_E^{(k)}\dvtx \mc
A^k \to\mc B$ are the $\mc B$-functionals defined by the \textit{free
moment-cumulant formula}
\[
E[a_1,\dotsc,a_n] = \sum_{\pi\in \mathit{NC}(n)} \kappa_E^{(\pi
)}[a_1,\dotsc,a_n].
\]
As above, this equation can be solved recursively for $\kappa_E^{(n)}$.
\end{longlist}
\end{definition}

While the definitions of conditional independence and freeness with
amalgamation given above appear at first to be quite different, they
have very similar expressions in terms of cumulants. In the free case,
the following theorem is due to Speicher~\cite{sp2}.

\begin{theorem}
Let $(\mc A, E\dvtx \mc A \to\mc B)$ be an operator-valued probability
space, and $(x_i)_{i \in I}$ a family of random variables in $\mc A$.
\begin{longlist}[(2)]
\item[(1)] If the algebra generated by $\mc B$ and $(x_i)_{i \in I}$ is
commutative, then the variables are conditionally independent given $B$
if and only if
\[
c_E^{(n)}[b_0x_{i_1}b_1,\dotsc,x_{i_n}b_n] = 0,
\]
whenever there are $1 \leq k,l \leq n$ such that $i_k \neq i_l$.
\item[(2)] The variables are free with amalgamation over $\mc B$ if and only if
\[
\kappa_E^{(n)}[b_0x_{i_1}b_1,\dotsc,x_{i_n}b_n] = 0,
\]
whenever there are $1 \leq k,l \leq n$ such that $i_k \neq i_l$.
\end{longlist}
\end{theorem}

Note that the condition in (1) is equivalent to the statement that if
$\pi\in P(n)$, then
\[
c_E^{(\pi)}[b_0x_{i_1}b_1,\dotsc,x_{i_n}b_n] = 0
\]
unless $\pi\leq\ker\mathbf i$, and likewise in (2) for $\pi\in
\mathit{NC}(n)$. Stronger characterizations of the joint distribution of
$(x_i)_{i \in I}$ can be given by specifying what types of partitions
may contribute nonzero cumulants.

\begin{theorem}\label{cumulantcharacterization}
Let $(\mc A,E\dvtx \mc A \to\mc B)$ be an operator-valued probability
space, and let $(x_i)_{i \in I}$ be a family of random variables in
$\mc A$.
\begin{longlist}[(2)]
\item[(1)] Suppose that $\mc B$ and $(x_i)_{i \in I}$ generate a commutative
algebra. The $\mc B$-valued joint distribution of $(x_i)_{i \in I}$ has
the property corresponding to $D$ in the table below if and only if for
any $\pi\in P(n)$
\[
c_E^{(\pi)}[b_0x_{i_1}b_1,\dotsc,x_{i_n}b_n] = 0
\]
unless $\pi\in D(n)$ and $\pi\leq\ker\mathbf i$.

\begin{center}
{\fontsize{9pt}{11pt}\selectfont{
\begin{tabular*}{\textwidth}{@{\extracolsep{\fill}}lc@{}}
\hline
\textbf{Partitions} $\bolds D$ & \textbf{Joint distribution}\\
\hline
$P$: All partitions & Independent\\
$P_h$: Partitions with even block sizes & Independent and even\\
$P_b$: Partitions with block size $\leq2$ & Independent Gaussian\\
$P_2$: Pair partitions & Independent centered Gaussian\\
\hline
\end{tabular*}}}
\end{center}\vspace*{9pt}
\item[(2)] The $\mc B$-valued joint distribution of $(x_i)_{i \in I}$ has
the property corresponding to $D$ in the the table below if and only if
for any $\pi\in \mathit{NC}(n)$
\[
\kappa_E^{(\pi)}[b_0x_{i_1}b_1,\dotsc,x_{i_n}b_n] = 0
\]
unless $\pi\in D(n)$ and $\pi\leq\ker\mathbf i$.

\vspace*{9pt}\begin{center}
{\fontsize{9pt}{11pt}\selectfont{
\begin{tabular*}{\textwidth}{@{\extracolsep{\fill}}lc@{}}
\hline
\textbf{Noncrossing partitions} $\bolds D$ & \textbf{Joint distribution}\\
\hline
$\mathit{NC}$: Noncrossing partitions & Freely independent\\
$\mathit{NC}_h$: Noncrossing partitions with even block sizes & Freely
independent and even\\
$\mathit{NC}_b$: Noncrossing partitions with block size $\leq2$ & Freely
independent semicircular\\
$\mathit{NC}_2$: Noncrossing pair partitions & Freely independent centered
semicircular\\
\hline
\end{tabular*}}}
\end{center}\vspace*{5pt}
\end{longlist}
\end{theorem}

\begin{pf}
These results are well known. In the classical case, note that the
results for $P_2,P_b$ are equivalent to the \textit{Wick formula} for
computing moments of independent Gaussian families. In the free case,
see~\cite{sp2,ns}.
\end{pf}

\begin{remark}
It is clear from the definitions that the classical and free cumulants
can be solved for from the joint moments. In fact, a combinatorial
formula for the cumulants in terms of the moments can be given. First
we recall the definition of the M\"{o}bius function on a partially
ordered set.
\end{remark}

\begin{definition}
Let $(P,<)$ be a finite partially ordered set. The \textit{M\"{o}bius
function} $\mu_P\dvtx P \times P \to\Z$ is defined by
\[
\mu_P(p,q) =
\cases{\displaystyle 0,  \qquad  p \not\leq q,\cr\displaystyle
1,  \qquad p = q,\vspace*{4pt}\cr\displaystyle
-1 + \sum_{l \geq1} (-1)^{l+1} \#\{(p_1,\dotsc,p_l) \in P^l\dvtx  p < p_1
< \dotsb< p_l < q\},   \vspace*{3pt}\cr
 \qquad p < q.
}
\]
\end{definition}

\begin{theorem}
Let $(\mc A, E\dvtx \mc A \to\mc B)$ be an operator-valued probability
space, and let $(x_i)_{i \in I}$ be a family of random variables.
Define the \textit{$\mc B$-valued moment functionals} $E^{(n)}$ by
\[
E^{(n)}[a_1,\dotsc,a_n] = E[a_1\dotsb a_n].
\]
\begin{longlist}[(2)]
\item[(1)] Suppose that $\mc B$ is commutative. Then for any $\sigma\in
P(n)$ and $a_1,\dotsc,\allowbreak a_n \in\mc A$, we have
\[
c_E^{(\sigma)}[a_1,\dotsc,a_n] = \mathop{\mathop{\sum}_{\pi\in P(n)}}_{ \pi
\leq\sigma} \mu_{P(n)}(\pi,\sigma) E^{(\pi)}[a_1,\dotsc,a_n].
\]
\item[(2)] For any $\sigma\in \mathit{NC}(n)$ and $a_1,\dotsc,a_n \in\mc A$, we have
\[
\kappa_E^{(\sigma)}[a_1,\dotsc,a_n] = \mathop{\mathop{\sum}_{\pi\in
\mathit{NC}(n)}}_{ \pi\leq\sigma} \mu_{\mathit{NC}(n)}(\pi,\sigma) E^{(\pi
)}[a_1,\dotsc,a_n].
\]
\end{longlist}
\end{theorem}

\begin{pf}
This follows from the \textit{M\"{o}bius inversion formula}; see \cite{sp2,ns}.~%
\end{pf}

\subsection*{Easy quantum groups} We will now briefly recall
some notions and results from~\cite{bsp}.

Consider a compact group $G \subset O_n$. By the Stone--Weierstrauss
theorem, $C(G)$ is generated by the $n^2$ coordinate functions $u_{ij}$
sending a matrix in $G$ to its $(i,j)$ entry. The structure of $G$ as a
compact group is captured by the commutative Hopf C$^*$-algebra $C(G)$
together with comultiplication, counit and antipode determined by
\begin{eqnarray*}
\Delta(u_{ij}) &=& \sum_{k=1}^n u_{ik} \otimes u_{kj},\\
\epsilon(u_{ij}) &=& \delta_{ij},\\
S(u_{ij}) &=& u_{ji}.
\end{eqnarray*}
Dropping the condition of commutativity, we obtain the following
definition, adapted from the fundamental paper of Woronowicz~\cite{wor1}.

\begin{definition}
An \textit{orthogonal Hopf algebra} is a unital C$^*$-algebra $A$
generated by $n^2$ self-adjoint elements $u_{ij}$, such that the
following conditions hold:
\begin{longlist}[(3)]
\item[(1)] The inverse of $u = (u_{ij}) \in M_n(A)$ is the transpose $u^t =
(u_{ji})$.
\item[(2)]$\Delta(u_{ij}) = \sum_{k} u_{ik} \otimes u_{kj}$ determines a
morphism $\Delta\dvtx A \to A \otimes A$.
\item[(3)]$\epsilon(u_{ij}) = \delta_{ij}$ defines a morphism $\epsilon
\dvtx A \to\C$.
\item[(4)]$S(u_{ij}) = u_{ji}$ defines a morphism $S\dvtx A \to A^{op}$.
\end{longlist}
\end{definition}

It follows from the definitions that $\Delta, \epsilon, S$ satisfy
the usual Hopf algebra axioms. If $A$ is an orthogonal Hopf algebra, we
use the heuristic formula ``$A = C(G)$,'' where $G$ is an \textit
{compact orthogonal quantum group}. Of course if $A$ is noncommutative,
then $G$ cannot exist as a concrete object, and all statements about
$G$ must be interpreted in terms of the Hopf algebra $A$.

The following two examples, constructed by Wang in~\cite{wang1,wang2},
are fundamental to our considerations.

\begin{definition}
{\baselineskip=0pt
\begin{longlist}[(2)]
\item[(1)]$A_o(n)$ is the universal C$^*$-algebra generated by $n^2$
self-adjoint elements $u_{ij}$, such that $u = (u_{ij}) \in
M_n(A_o(n))$ is orthogonal.
\item[(2)]$A_s(n)$ is the universal C$^*$-algebra generated by $n^2$
projections $u_{ij}$, such that the sum along any row or column of $u =
(u_{ij}) \in M_n(A_s(n))$ is the identity.
\end{longlist}}
\end{definition}

As discussed above, we use the notation $A_o(n) = C(O_n^+)$, $A_s(n) =
C(S_n^+)$, and call $O_n^+$ and $S_n^+$ the \textit{free orthogonal
group} and \textit{free permutation group}, respectively.

We now recall the ``easiness'' condition from~\cite{bsp} for a compact
orthogonal quantum group $S_n \subset G \subset O_n^+$. Let $u,v$ be
the fundamental representations of $G,S_n$ on $\C^n$, respectively. By
functoriality, the space $\Hom(u^{\otimes k},u^{\otimes l})$ of
intertwining operators is contained in $\Hom(v^{\otimes k}, v^{\otimes
l})$ for any $k,l$. But the Hom-spaces for $v$ are well known: they are
spanned by operators $T_\pi$ with $\pi$ belonging to the set $P(k,l)$
of partitions between $k$ upper and $l$ lower points. Explicitly, if
$e_1,\dotsc,e_n$ denotes the standard basis of $\C^n$, then the
formula for $T_\pi$ is given by
\[
T_\pi(e_{i_1} \otimes\dotsb\otimes e_{i_k}) = \sum_{j_1,\dotsc
,j_l} \delta_\pi
\pmatrix{\displaystyle i_1 \dotsb i_k
\cr\displaystyle
  j_1\dotsb j_l
}
e_{j_1} \otimes\dotsb\otimes e_{j_l}.
\]
Here the $\delta$ symbol appearing on the right-hand side is 1 when
the indices ``fit,'' that is, if each block of $\pi$ contains equal
indices, and 0 otherwise.

It follows from the above discussion that $\Hom(u^{\otimes k},
u^{\otimes l})$ consists of certain linear combinations of the
operators $T_\pi$, with $\pi\in P(k,l)$. We call $G$ ``easy'' if
these spaces are spanned by partitions.

\begin{definition}
A compact orthogonal quantum group $S_n \subset G \subset O_n^+$ is
called \textit{easy} if for each $k,l \in\N$, there exist sets
$D(k,l) \subset P(k,l)$ such that $\Hom(u^{\otimes k}, u^{\otimes l})
= \mathrm{span}(T_\pi\dvtx \pi\in D(k,l))$. If we have $D(k,l) \subset
\mathit{NC}(k,l)$ for each $k,l \in\N$, we say that $G$ is a \textit{free
quantum group}.
\end{definition}

There are four natural examples of classical groups which are easy:
\vspace*{10pt}\begin{center}
{\fontsize{9pt}{11pt}\selectfont{
\begin{tabular}{@{}lc@{}}
\hline
\textbf{Group} & \textbf{Partitions}\\
\hline
Permutation group $S_n$ & $P$: All partitions\\
Orthogonal group $O_n$ & $P_2$: Pair partitions\\
Hyperoctahedral group $H_n$ & $P_h$: Partitions with even block sizes\\
Bistochastic group $B_n$ & $P_b$: Partitions with block size $\leq2$\\
\hline
\end{tabular}}}
\end{center}\vspace*{10pt}
There are also the two trivial modifications $S_n' = S_n \times\Z_2$
and $B_n' = B_n \times\Z_2$, and it was shown in~\cite{bsp} that
these six examples are the only ones.\vadjust{\goodbreak}

There is a one-to-one correspondence between classical easy groups and
free quantum groups, which on a combinatorial level corresponds to
restricting to noncrossing partitions:
\vspace*{10pt}
\begin{center}
{\fontsize{9pt}{11pt}\selectfont{
\begin{tabular}{@{}lc@{}}
\hline
\textbf{Quantum group} & \textbf{Partitions}\\
\hline
$S_n^+$ & $\mathit{NC}$: All noncrossing partitions\\
$O_n^+$ & $\mathit{NC}_2$: Noncrossing pair partitions\\
$H_n^+$ & $\mathit{NC}_h$: Noncrossing partitions with even block
sizes\\
$B_n^+$ & $\mathit{NC}_b$: Noncrossing partitions with block size $\leq2$\\
\hline
\end{tabular}}}
\end{center}\vspace*{10pt}
There are also free versions of $S_n',B_n'$, constructed in~\cite{bsp}.

In general, the class of easy quantum groups appears to be quite rigid
(see~\cite{bcs1} for a discussion here). However, two more examples
can be obtained as ``half-liberations.'' The idea is that instead of
removing the commutativity relations from the generators $u_{ij}$ of
$C(G)$ for a classical easy group $G$, which would produce $C(G^+)$, we
instead require that the the generators ``half-commute,'' that is, $abc
= cba$ for $a,b,c \in\{u_{ij}\}$. More precisely, we define $C(G^*) =
C(G^+)/I$, where $I$ is the ideal generated by the relations $abc =
cba$ for $a,b,c \in\{u_{ij}\}$. For $G = S_n,S_n',B_n,B_n'$ we have
$G^* = G$, however for $O_n,H_n$, we obtain new quantum groups
$O_n^*,H_n^*$. The corresponding partition categories $E_2,E_h$ consist
of all pair partitions, respectively all partitions, which are \textit
{balanced} in the sense that each block contains as many odd as even legs.

\section{Half independence}\label{sec:halfindependence}

In this section, we introduce a new kind of independence which appears
in the de Finetti theorems for the half-liberated quantum groups $H^*$
and $O^*$. To define this notion, we require that the variables have a
certain degree of commutativity.

\begin{definition}\label{halfcommutedef}
Let $(x_i)_{i \in I}$ be a family of noncommutative random variables.
We say that the variables \textit{half-commute} if
\[
x_ix_jx_k = x_kx_jx_i
\]
for all $i,j,k \in I$.
\end{definition}

Observe that if $(x_i)_{i \in I}$ half-commute, then in particular
$x_i^2$ commutes with $x_j$ for any $i,j \in I$.

\begin{definition}
Let $(\mc A,E\dvtx \mc A \to\mc B)$ be an operator-valued probability
space, and suppose that $\mc B$ is contained in the center of $\mc A$.
Let $(x_i)_{i \in I}$ be a~family of random variables in $\mc A$ which
half-commute. We say that $(x_i)_{i \in I}$ are \textit{conditionally
half-independent given $\mc B$}, or \textit{half-independent with
respect to $E$}, if the following conditions are satisfied:
\begin{longlist}[(2)]
\item[(1)] The variables $(x_i^2)_{i \in I}$ are conditionally independent
given $\mc B$.\vadjust{\goodbreak}
\item[(2)] For any $i_1,\dotsc,i_k \in I$, we have
\[
E[x_{i_1}\dotsb x_{i_k}] = 0
\]
unless for each $i \in I$ the set of $1 \leq j \leq k$ such that $i_j =
i$ contains as many odd as even numbers, that is, unless $\ker\mathbf
i$ is balanced.
\end{longlist}
If $B = \C$, then the variables are simply called
\textit{half-independent}.\vspace*{-2pt}
\end{definition}

\begin{remark}
As a first remark, we note that half-independence is defined only
between random variables and not at the level of algebras, in contrast
with classical and free independence. In fact, it is known from \cite
{sp1} there are no other good notions of independence between unital
algebras other than classical and free.\vspace*{-2pt}
\end{remark}

The conditions may appear at first to be somewhat artificial, but are
motivated by the following natural example.\vspace*{-2pt}

\begin{example}\label{halfexample} Let $(\Omega, \Sigma,\mu)$ be a
(classical) probability space, and let $L(\mu)$ denote the algebra of
complex-valued random variables on $\Omega$ with finite moments of all orders.
\begin{longlist}[(2)]
\item[(1)] Let $(\xi_i)_{i \in I}$ be a family of independent random
variables in $L(\mu)$. Suppose that for each $i \in I$, the
distribution of $\xi_i$ is such that
\[
\mb E[\xi_i^n \overline{\xi_i^m}] = 0
\]
unless $n = m$. Define random variables $x_i$ in $(M_2(L(\mu)), \mb E
\circ\tr)$ by
\[
x_i =
\pmatrix{\displaystyle
0 & \xi_i\cr\displaystyle  \overline\xi_i & 0
}
.
\]
A simple computation shows that the variables $(x_i)_{i \in I}$
half-commute. Since
\[
x_i^2 = |\xi_i|^2 I_2,
\]
it is clear that $(x_i^2)_{i \in I}$ are independent with respect to
$\mb E \circ\tr$. Moreover, the assumption on the distributions of
the $\xi_i$ clearly implies that $\mb E[\tr[x_{i_1}\dotsb x_{i_k}]] =
0$ unless $k$ is even and $\ker\mathbf i$ is balanced. So $(x_i)_{i
\in I}$ are half-independent.

Observe also that the distribution of $x_i$ is equal to that of $(\xi
_i\overline{\xi_i})^{1/2}$, where the square root is chosen such that
the distribution is even. We call this the \textit{squeezed version}
of the complex distribution $\xi_i$ (cf.~\cite{bsp}).

\item[(2)] Of particular interest is the case that the $(\xi_i)_{i \in I}$
have complex Gaussian distributions. Here the distribution of $x_i$ is
the squeezed version of the complex Gaussian $\xi_i$, which is a
symmetrized Rayleigh distribution.\vspace*{-2pt}
\end{longlist}
\end{example}

\begin{remark}
We will show in Proposition~\ref{halfindependentform} below that any
half-independ\-ent family can be modeled as in the example above. First,
we will show that, as for classical and free independence, the joint
distribution of a family of half-independent random variables $(x_i)_{i
\in I}$ is determined by the distributions of $x_i$ for $i \in I$. It
is convenient to first introduce the following family of permutations
which are related to the half-commutation relation.\vadjust{\goodbreak}
\end{remark}

\begin{definition}
We say that a permutation $\omega\in S_n$ \textit{preserves parity}
if $\omega(i) \equiv i~(\operatorname{mod} 2)$ for $1 \leq i \leq n$.
\end{definition}

The collection of parity preserving partitions in $S_n$ clearly form a
subgroup, which is simply $S(\{1,3,\dotsc\}) \times S(\{2,4,\dotsc\}
)$. Moreover, this subgroup is generated by the transpositions $(i\;
i+2)$ for $1 \leq i \leq n-2$. It follows that if $(x_i)_{i \in I}$
half-commute, then
\[
x_{i_1}\dotsb x_{i_n} = x_{i_{\omega(1)}}\dotsb x_{i_{\omega(n)}},
\]
whenever $\omega\in S_n$ preserves parity.

\begin{lemma}\label{halfjoint}
Let $(\mc A,E\dvtx \mc A \to\mc B)$ be an operator-valued probability space
such that $\mc B$ is contained in the center of $\mc A$. Suppose that
$(x_i)_{i \in I}$ is a family of random variables in $\mc A$ which are
conditionally half-independent given $\mc B$. Then the $\mc B$-valued
joint distribution of $(x_i)_{i \in I}$ is uniquely determined by the
$\mc B$-valued distributions of $x_i$ for $i \in I$.
\end{lemma}

\begin{pf}
Let $i_1,\dotsc,i_k \in I$. We know that
\[
E[x_{i_1}\dotsb x_{i_k}] = 0
\]
unless we have that for each $i \in I$, the set of $1 \leq j \leq k$
such that $i_j =i$ has as many odd as even elements. So suppose that
this the case. By the remark above, we know that $x_{i_1}\dotsb x_{i_k}
= x_{i_{\omega(1)}}\dotsb x_{i_{\omega(k)}}$ whenever $\omega\in
S_k$ is parity preserving. With an appropriate choose of $\omega$, it
follows that
\[
x_{i_1}\dotsb x_{i_k} = x_{j_1}^{2(k_1)}\dotsb x_{j_m}^{2(k_m)}
\]
for some $j_1,\dotsc,j_m \in I$ and $k_1,\dotsc,k_m \in\N$ such
that $k = 2(k_1 + \dotsb+ k_m)$. Since the joint distribution of
$(x_i^2)_{i \in I}$ is clearly determined by the distributions of $x_i$
for $i \in I$, the result follows.
\end{pf}

\begin{proposition}\label{halfindependentform}
Let $(x_i)_{i \in I}$ be a half-commuting family of random variables in
a W$^*$-probability space $(M,\varphi)$ which are half-independent.
Then there are independent complex-valued random variables $(\xi_i)_{i
\in I}$ such that $\mb E[\xi_i^n \overline{\xi_i^m}] = 0$ unless $n
= m$, and such that $(x_i)_{i \in I}$ has the same joint distribution
as the family $(y_i)_{i \in I}$,
\[
y_i =
\pmatrix{\displaystyle 0 & \xi_i\cr\displaystyle  \overline\xi_i & 0
}
.
\]
\end{proposition}

\begin{pf}
Let $(X_i)_{i \in I}$ be a family of independent random variables such
that $X_i$ has the same distribution as $x_i$. Let $(U_i)_{i \in I}$ be
a family of independent Haar unitary random variables which are
independent from $(X_i)_{i \in I}$, and let $\xi_i = U_iX_i$. Then
$(\xi_i)_{i \in I}$ are independent and
\[
\mb E[\xi_i^n\overline{\xi_i^m}] = \mb E[X_i^{n+m}]\mb
E[U_i^n\overline{U_i}^m] = \delta_{nm}\varphi(x_i^{2n}).\vadjust{\goodbreak}
\]
From Example~\ref{halfexample}, the variables $(y_i)_{i \in I}$
defined by
\[
y_i =
\pmatrix{\displaystyle 0 & \xi_i\cr\displaystyle  \overline\xi_i & 0
}
\]
are half-independent, and $y_i$ has the same distribution as $x_i$ for
each $i \in I$. By Lemma~\ref{halfjoint}, $(y_i)_{i \in I}$ has the
same joint distribution as $(x_i)_{i \in I}$.
\end{pf}

\begin{remark}
We have stated our results in the scalar case $\mc B = \C$ for
simplicity, but note that with suitable modifications, Example \ref
{halfexample} and Proposition~\ref{halfindependentform} apply equally
well to conditionally half-independent families.
\end{remark}

We will now develop a combinatorial theory for half-independence, based
on the family $E_h$ of balanced partitions.

\begin{definition}
Let $(\mc A,E\dvtx \mc A \to\mc B)$ be an operator-valued probability
space, and suppose that $\mc B$ is contained in the center of $\mc A$.
Let $(x_i)_{i \in I}$ be a~family of random-variables in $\mc A$, and
suppose that
\[
E[x_{i_1}\dotsb x_{i_k}] = 0
\]
for any odd $k$ and $i_1,\dotsc,i_k \in I$. Define the \textit
{half-liberated cumulants} $\gamma_E^{(n)}$ by the \textit
{half-liberated moment-cumulant formula}
\[
E[x_{i_1}\dotsb x_{i_k}] = \mathop{\mathop{\sum}_{\pi\in E_h(k)}}
_{ \pi\leq \ker\mathbf i} \gamma_E^{(\pi)}[x_{i_1},\dotsc,x_{i_k}],
\]
where $\gamma_E^{(\pi)}[x_{i_1},\dotsc,x_{i_k}]$ is defined, as in
the classical case, by the formula
\[
\gamma_E^{(\pi)}[x_{i_1},\dotsc,x_{i_k}] = \prod_{V \in\pi}
\gamma_E^{(|V|)}(V)[x_{i_1},\dotsc,x_{i_k}].
\]
\end{definition}

Observe that both sides of the moment-cumulant formula above are equal
to zero for odd values of $k$, and for even values the right hand side
is equal to $\gamma_E^{(k)}[x_{i_1},\dotsc,x_{i_k}]$ plus products of
lower ordered terms and hence $\gamma_E^{(k)}$ may be solved for
recursively. As in the free and classical cases, we may apply the M\"
{o}bius inversion formula to obtain the following equation for~$\gamma
_E^{(\pi)}$, $\pi\in E_h(k)$:
\[
\gamma_E^{(\pi)}(x_{i_1},\dotsc,x_{i_k}) = \mathop{\mathop{\sum}_{\sigma
\in E_h(k)}}_{ \sigma\leq\pi} \mu_{E_h(k)}(\sigma,\pi) E^{(\pi
)}[x_{i_1},\dotsc,x_{i_k}].
\]

\begin{theorem}\label{halfcumulantcharacterization}
Let $(\mc A,E\dvtx \mc A \to\mc B)$ be an operator-valued probability
space, and suppose that $\mc B$ is contained in the center of $\mc A$.
Suppose $(x_i)_{i \in\N}$ is a family of variables in $\mc A$ which
half-commute. Then the following conditions are equivalent:
\begin{longlist}[(2)]
\item[(1)]$(x_i)_{i \in\N}$ are half-independent with respect to $E$.
\item[(2)]$E[x_{i_1}\dotsb x_{i_k}] = 0$ whenever $k$ is odd, and
\[
\gamma^{(\pi)}_{E}[x_{i_1},\dotsc,x_{i_k}] = 0
\]
for any $\pi\in E_h(k)$ such that $\pi\not\leq\ker\mathbf i$.
\end{longlist}
\end{theorem}

\begin{pf}
First, suppose that condition (2) holds. From the moment-cumulant
formula, we have
\[
E[x_{i_1}\dotsb x_{i_k}] = \mathop{\mathop{\sum}_{\pi\in E_h(k)}}_{ \pi\leq
\ker\mathbf i} \gamma_E^{(\pi)}[x_{i_1},\dotsc,x_{i_k}]
\]
for any $k \in\N$ and $i_1,\dotsc,i_k \in I$. Observe that if $\ker
\mathbf i$ is not balanced then there is no $\pi\in E_h(k)$ such that
$\pi\leq\ker\mathbf i$, so it follows that $E[x_{i_1}\dotsb x_{i_k}]
= 0$. It remains to show that $(x_i^2)_{i \in I}$ are independent.
Choose $k_1,\dotsc,k_m \in\N$, distinct $i_1,\dotsc,i_m \in I$ and
let $k = 2(k_1 + \dotsb+ k_m)$. Let $\tau\in E_h(k)$ be the partition
with blocks $\{1,\dotsc,2k_1\},\dotsc,\{2(k_1+\dotsb+
k_{m-1})+1,\dotsc,2k\}$. Then
\begin{eqnarray*}
E \bigl[x_{i_1}^{(2k_1)}\dotsb x_{i_m}^{(2k_m)} \bigr] &=& \mathop{\mathop{\sum}_{\pi\in E_h(k)}}_{ \pi\leq\tau} \gamma_E^{(\pi
)}[x_{i_1},\dotsc,x_{i_1},x_{i_2},\dotsc,x_{i_m},\dotsc,x_{i_m}]\\
&=& \prod_{1 \leq j \leq m} \sum_{\pi\in E_h(2k_j)}
\gamma_E^{(\pi)}[x_{i_j},\dotsc,x_{i_j}]\\
&=& \prod_{1 \leq j \leq m} E\bigl[x_{i_j}^{(2k_j)}\bigr],
\end{eqnarray*}
so that $(x_i^2)_{i \in I}$ are independent and hence $(x_i)_{i \in I}$
are half-independent.

The implication $(1)\Rightarrow(2)$ actually follows from
$(2)\Rightarrow(1)$. Indeed, suppose that $(x_i)_{i \in I}$ are
half-independent. Consider the algebra $\mc A' = \mc B \langle y_i\dvtx  i
\in I \rangle/\langle y_iy_jy_k = y_ky_jy_i\rangle$ of polynomials
in half-commuting indeterminates $(y_i)_{i \in I}$ and coefficients in
$\mc B$. Define a conditional expectation $E'\dvtx \mc A' \to\mc B$ by
\[
E'[y_{i_1}\dotsb y_{i_k}] = \mathop{\mathop{\sum}_{\pi\in E_h(k)}}_{ \pi\leq
\ker\mathbf i} \gamma_{E}^{(\pi)}[x_{i_1},\dotsc,x_{i_k}].
\]
(It is easy to see that $E'$ is well defined, that is, compatible with
the half-commutation relations.) Since the half-liberated cumulants are
uniquely determined by the moment-cumulant formula, it follows that
\[
\gamma_{E'}^{(\pi)}[y_{i_1},\dotsc,y_{i_k}] =
\cases{\displaystyle \gamma_E^{(\pi)}[x_{i_1},\dotsc,x_{i_k}],  &\quad$\pi\leq
\ker\mathbf i$,\cr\displaystyle 0,  &\quad  otherwise.}
\]
By the first part, it follows that $(y_i)_{i \in I}$ are
half-independent with respect to $E'$. Since $y_i$ has the same $\mc
B$-valued distribution as $x_i$, it follows from
Lemma~\ref{halfjoint}\vadjust{\goodbreak}
that $(y_i)_{i \in I}$ have the same joint distribution as $(x_i)_{i
\in I}$. It then follows from the moment-cumulant formula that these
families have the same half-liberated cumulants, and hence $\gamma
_E^{(\pi)}[x_{i_1},\dotsc,x_{i_k}] = 0$ unless $\pi\leq\ker\mathbf
i$.\vspace*{-1pt}
\end{pf}

Recall that (centered) Gaussian and semicircular distributions are
characterized by the property that their nonvanishing cumulants are
those corresponding to pair and noncrossing pair partitions,
respectively. We will now show that for half-independence, it is the
symmetrized Rayleigh distribution which has this property. This follows
from the considerations in~\cite{bsp}, but we include here a direct proof.\vspace*{-1pt}

\begin{proposition}\label{halfgausscharacterization}
Let $x$ be a random variable in $(\mc A,\varphi)$ which has an even
distribution. Then $x$ has a symmetrized Rayleigh distribution if and
only if
\[
\gamma_E^{(\pi)}[x,\dotsc,x] = 0
\]
for any $\pi\in E_h(k)$ such that $\pi\notin E_2(k)$.\vspace*{-1pt}
\end{proposition}

\begin{pf}
Since the distribution of $x$ is determined uniquely by its
half-liberated cumulants, it suffices to show that if the cumulants
have the stated property then $x$ has a symmetrized Rayleigh
distribution. Suppose that this is the case, then
\begin{eqnarray*}
\varphi(x^{k}) &=& \sum_{\pi\in E_2(k)} \gamma^{(\pi)}[x,\dotsc
,x]\\
&=& \gamma^{(2)}[x,x] \#\{\pi\in E_2(k)\}.
\end{eqnarray*}
It is easy to see that the number of partitions in $E_2(k)$ is $m!$ if
$k=2m$ is even and is zero if $k$ is odd. Since these agree with the
moments of a symmetrized Rayleigh distribution, the result follows.\vspace*{-1pt}
\end{pf}

\section{Weingarten estimate}\label{sec:weingarten}

It is a fundamental result of Woronowicz~\cite{wor1} that if $G$ is a
compact orthogonal quantum group, then there is a unique state \mbox{$\int
\dvtx C(G) \to\C$}, called the \textit{Haar state}, which is left and
right invariant in the sense that
\[
\biggl(\int{}\otimes{}\mathrm{id}\biggr)\Delta(f) = \int
(f)\cdot1_{C(G)} = \biggl(\mathrm{id} \otimes\int\biggr)\Delta(f)
\qquad\bigl(f \in C(G)\bigr).
\]
If $G \subset O_n$ is a compact group, then the Haar state on $C(G)$ is
given by integrating against the Haar measure on $G$.

One quite useful aspect of the easiness condition for a compact
orthogonal quantum group is that it leads to a combinatorial \textit
{Weingarten formula} for computing the Haar state, which we now recall
from~\cite{bsp}.\vspace*{-1pt}

\begin{definition}
Let $D(k) \subset P(k)$ be a collection of partitions. For \mbox{$n \in\N$},
define the \textit{Gram matrix} $(G_{kn}(\pi,\sigma))_{\pi,\sigma
\in D(k)}$ by the formula
\[
G_{kn}(\pi,\sigma) = n^{|\pi\vee\sigma|}.\vadjust{\goodbreak}
\]
$G_{kn}$ is invertible for $n$ sufficiently large (see Proposition \ref
{West}), define the \textit{Weingarten matrix} $W_{kn}$ to be its inverse.
\end{definition}

\begin{theorem} \label{weingarten}
Let $G \subset O_n^+$ be an easy quantum group and let $D(k) \subset
P(0,k)$ be the corresponding collection of partitions having no upper
points. If $G_{kn}$ is invertible, then
\[
\int u_{i_1j_1}\dotsb u_{i_kj_k} = \mathop{\mathop{\mathop{\sum}_{\pi,\sigma\in
D(k)}}_{\pi\leq\ker\mathbf i}}_{ \sigma\leq\ker\mathbf j}
W_{kn}(\pi,\sigma).
\]
\end{theorem}

\begin{remark}
The statement of the theorem above is from~\cite{bsp}, but goes back
to work of Weingarten~\cite{wein} and was developed in a series of
papers~\cite{col,cos,bc1,bc2}. Note that this reduces the problem of
evaluating integrals over $G$ to computing the entries of the
Weingarten matrix. We will now give an estimate on the asymptotic
behavior of $W_{kn}$ as $n \to\infty$. This unifies and extends the
estimates given in~\cite{bc1} and~\cite{cur3} for $O^+, S^+$.
\end{remark}

\begin{proposition}\label{West}
Let $k \in\N$ and $D(k) \subset P(k)$. For $n$ sufficiently large,
the Gram matrix $G_{kn}$ is invertible. Moreover, the entries of the
Weingarten matrix $W_{kn}=G_{kn}^{-1}$ satisfy the following:
\begin{longlist}[(2)]
\item[(1)] $W_{kn}(\pi,\sigma) = O(n^{|\pi\vee\sigma| - |\pi| -
|\sigma|})$.
\item[(2)] If $\pi\leq\sigma$, then
\[
n^{|\pi|}W_{kn}(\pi,\sigma) = \mu_{D(k)}(\pi,\sigma) + O(n^{-1}),
\]
\end{longlist}
where $\mu_{D(k)}$ is the M\"{o}bius function on the partially ordered
set $D(k)$ under the restriction of the order on $ P(k)$.
\end{proposition}

\begin{pf}
We use a standard method from~\cite{col,cos}, further developed in~\cite{bc1,bc2,cur1}.

First, note that
\[
G_{kn} = \Theta_{kn}^{1/2}(1 + B_{kn})\Theta_{kn}^{1/2},
\]
where
\begin{eqnarray*}
\Theta_{kn}(\pi,\sigma) &=&
\cases{\displaystyle
n^{|\pi|} ,&\quad $\pi= \sigma$,  \cr\displaystyle
0 ,&\quad $\pi\neq\sigma$,
}
\\
B_{kn}(\pi,\sigma) &=&
\cases{\displaystyle
0 ,&\quad $\pi= \sigma$,
\vspace*{2pt}\cr\displaystyle
n^{|\pi\vee\sigma| - \fracb{|\pi| + |\sigma|}{2}} ,&\quad $\pi\neq\sigma$.
}
\end{eqnarray*}
Note that the entries of $B_{kn}$ are $O(n^{-1/2})$, it follows that
for $n$ sufficiently large $1 + B_{kn}$ is invertible and
\[
(1 + B_{kn})^{-1} = 1 - B_{kn} + \sum_{l \geq1} (-1)^{l+1} B_{kn}^{l+1}.\vadjust{\goodbreak}
\]
$G_{kn}$ is then invertible, and
\begin{eqnarray*}
 W_{kn}(\pi,\sigma)  &=& \sum_{l \geq1} (-1)^{l+1} (\Theta
_{kn}^{-1/2}B_{kn}^{l+1}\Theta_{kn}^{-1/2})(\pi,\sigma)\\
&&{} +
\cases{\displaystyle n^{-|\pi|} ,&\quad $\pi= \sigma$,
\cr\displaystyle
-n^{|\pi\vee\sigma| -|\pi| - |\sigma|} ,&\quad $\pi\neq\sigma$.
}
\end{eqnarray*}
Now for $l \geq1$ we have
\begin{eqnarray*}
&&(\Theta_{kn}^{-1/2}B_{kn}^{l+1}\Theta_{kn}^{-1/2})(\pi,\sigma)\\
&& \qquad  =
\mathop{\mathop{\sum}_{\nu_1,\dotsc,\nu_l \in D(k)}}_{ \pi\neq\nu_1 \neq
\dotsb\neq\nu_l \neq\sigma} n^{|\pi\vee\nu_1| + |\nu_1 \vee
\nu_2| + \dotsb+ |\nu_l \vee\sigma| - |\nu_1| - \dotsb- |\nu_l|
- |\pi| - |\sigma|}.
\end{eqnarray*}
So to prove (1), it suffices to show that if $\nu_1,\dotsc,\nu_l \in
D(k)$, then
\[
|\pi\vee\nu_1| + |\nu_1 \vee\nu_2| + \dotsb+ |\nu_l \vee\sigma
| \leq|\pi\vee\sigma| + |\nu_1| + \dotsb+ |\nu_l|.
\]

We will use the fact that $P(k)$ is a \textit{semi-modular lattice}
(\cite{birk}, Section~I.8, Example~9):
If $\nu,\tau\in\eu{P}(k)$, then
\[
|\nu| + |\tau| \leq|\nu\vee\tau| + |\nu\wedge\tau|.
\]
We will now prove the claim by induction on $l$, for $l = 1$ we may
apply the formula above to find
\begin{eqnarray*}
|\pi\vee\nu| + |\nu\vee\sigma| &\leq&|(\pi\vee\nu) \vee(\nu
\vee\sigma)| + |(\pi\vee\nu) \wedge(\nu\vee\sigma)|\\
&\leq&|\pi\vee\sigma| + |\nu|.
\end{eqnarray*}
Now let $l > 1$, by induction we have
\[
|\pi\vee\nu_1| + |\nu_1 \vee\nu_2| + \dotsb+ |\nu_{l-1} \vee
\nu_l| \leq|\pi\vee\nu_l| + |\nu_1| + \dotsb+ |\nu_{l-1}|.
\]
Also $|\nu_l \vee\sigma| \leq|\pi\vee\sigma| + |\nu_l| - |\pi
\vee\nu_l|$, and the result follows.

To prove (2), suppose $\pi, \sigma\in D(k)$ and $\pi\leq\sigma$.
The terms which contribute to order $n^{-|\pi|}$ in the expansion come
from sequences $\nu_1,\dotsc,\nu_l \in D(k)$ such that $\pi\neq\nu
_1 \neq\dotsb\neq\nu_l \neq\sigma$ and
\[
|\pi\vee\nu_1| + \dotsb+ |\nu_l \vee\sigma| = |\sigma| + |\nu
_1| + \dotsb+ |\nu_l|.
\]
Since $|\pi\vee\nu_1| \leq|\nu_1|$, $|\nu_1 \vee\nu_2| \leq
|\nu_2|,\dotsc,|\nu_l \vee\sigma| \leq\sigma$, it follows that
each of these must be an equality, which implies $\pi< \nu_1 < \dotsb
< \nu_l < \sigma$. Conversely, any $\nu_1,\dotsc,\nu_l \in D(k)$
such that $\pi< \nu_1 < \dotsb< \nu_l < \sigma$ clearly satisfy
this equation. Therefore, the coefficient of $n^{-|\pi|}$ in
$W_{kn}(\pi,\sigma)$ is
\[
\cases{\displaystyle
1,  &\quad $\pi= \sigma$,\cr\displaystyle
-1 + \sum_{l=1}^\infty(-1)^{l+1} \#\{(\nu_1,\dotsc,\nu_l) \in
D(k)^l\dvtx  \pi< \nu_1 < \dotsb< \nu_l < \sigma\},  &\quad $\pi<
\sigma$,
}
\]
which is precisely $\mu_{D(k)}(\pi,\sigma)$.\vadjust{\goodbreak}
\end{pf}

Recall that the free, half-liberated and classical cumulants are
obtained from moment functionals by using the M\"{o}bius functions on
$\mathit{NC},E_h$ and $P$, respectively. To show that this is compatible with
Proposition~\ref{West}, we will need the following result.

\begin{proposition}\label{moebius}
{\baselineskip=0pt\begin{longlist}[(3)]
\item[(1)] If $D = \mathit{NC},\mathit{NC}_2,\mathit{NC}_b,\mathit{NC}_h$, then
\[
\mu_{D(k)}(\pi,\sigma) = \mu_{\mathit{NC}(k)}(\pi,\sigma)
\]
for all $\pi,\sigma\in D(k)$.
\item[(2)] If $D = E_h, E_2$, then
\[
\mu_{D(k)}(\pi,\sigma) = \mu_{ E_h(k)}(\pi,\sigma)
\]
for all $\pi,\sigma\in D(k)$.
\item[(3)] If $D = P, P_2, P_b, P_h$, then
\[
\mu_{D(k)}(\pi,\sigma) = \mu_{ P(k)}(\pi,\sigma)
\]
for all $\pi,\sigma\in D(k)$.
\end{longlist}}
\end{proposition}

\begin{pf}
Let $Q = \mathit{NC},E_h,P$ according to cases (1), (2), (3). It is easy to see
in each case that $D(k)$ is closed under taking intervals in $Q(k)$,
that is, if $\pi_1,\pi_2 \in D(k)$, $\sigma\in Q(k)$ and $\pi_1 <
\sigma< \pi_2$ then $\sigma\in D(k)$. The result now follows
immediately from the definition of the M\"{o}bius function.
\end{pf}

\section{Finite quantum invariant sequences}\label{sec:finite}
We begin this section by defining the notion of quantum invariance for
a sequence of noncommutative random variables under ``transformations''
coming from an orthogonal quantum group $G_n \subset O_n^+$.

Let $\ms P_n = \C\langle t_1,\dotsc,t_n \rangle$, and let $\alpha
_n\dvtx \ms P_n \to\ms P_n \otimes C(G_n)$ be the unique unital
homomorphism such that
\[
\alpha_n(t_j) = \sum_{i = 1}^n t_i \otimes u_{ij}.
\]
It is easily verified that $\alpha_n$ is an action of $G_n$, that is,
\[
(\mathrm{id} \otimes\Delta) \circ\alpha_n = (\alpha_n \otimes
\mathrm{id}) \circ\alpha_n
\]
and
\[
(\mathrm{id} \otimes\epsilon) \circ\alpha_n = \mathrm{id}.
\]

\begin{definition}
Let $(x_1,\dotsc,x_n)$ be a sequence of random variables in
a~noncommutative probability space $(\mc B,\varphi)$. We say that the
joint distribution of this sequence is \textit{invariant under $G_n$},
or that the sequence is \textit{$G_n$-invariant}, if the distribution
functional $\varphi_x\dvtx \ms P_n \to\C$ is invariant under the coaction
$\alpha_n$, that is,
\[
(\varphi_x \otimes\mathrm{id}) \alpha_n(p) = \varphi_x(p)
\]
for all $p \in\ms P_n$. More explicitly, the sequence $(x_1,\dotsc
,x_n)$ is $G_n$-invariant if
\[
\varphi(x_{j_1}\dotsb x_{j_k})1_{C(G_n)} = \sum_{1 \leq i_1,\dotsc
,i_k \leq n} \varphi(x_{i_1}\dotsb x_{i_k})u_{i_1j_1}\dotsb u_{i_kj_k}
\]
as an equality in $C(G_n)$, for all $k \in\N$ and $1 \leq j_1,\dotsc
,j_k \leq n$.
\end{definition}

\begin{remark}
Suppose that $G_n \subset O_n$ is a compact group. By evaluating both
sides of the above equation at $g \in G_n$, we see that a sequence
$(x_1,\dotsc,x_n)$ is $G_n$-invariant if and only if
\[
\varphi(x_{j_1}\dotsb x_{j_k}) = \sum_{1 \leq i_1,\dotsc,i_k \leq n}
g_{i_1j_1}\dotsb g_{i_kj_k}\varphi(x_{i_1}\dotsb x_{i_k})
\]
for each $k \in\N$, $1 \leq j_1,\dotsc,j_k \leq n$ and $g = (g_{ij})
\in G_n$, which coincides with the usual notion of $G_n$-invariance for
a sequence of classical random variables.
\end{remark}

We will now prove a converse to Theorem~\ref{definetti}, which holds
for finite sequences and in a purely algebraic context. The proof is
adapted from the method of~\cite{ksp}, Proposition 3.1.

\begin{proposition}\label{easydirection}
Let $(\mc A,\varphi)$ be a noncommutative probability space, $1 \in
\mc B \subset\mc A$ a unital subalgebra and $E\dvtx \mc A \to\mc B$ a
conditional expectation which preserves $\varphi$. Let $(x_1,\dotsc
,x_n)$ be a sequence in $\mc A$.
\begin{enumerate}[(3)]
\item[(1)] Free case:
\begin{longlist}[(b)]
\item[(a)] If $x_1,\dotsc,x_n$ are freely independent and identically
distributed with amalgamation over $\mc B$, then the sequence is
$S_n^+$-invariant.
\item[(b)] If $x_1,\dotsc,x_n$ are freely independent and identically
distributed with amalgamation over $\mc B$, and have even distributions
with respect to~$E$, then the sequence is $H_n^+$-invariant.
\item[(c)] If $x_1,\dotsc,x_n$ are freely independent and identically
distributed with amalgamation over $\mc B$, and have semicircular
distributions with respect to $E$, then the sequence is $B_n^+$-invariant.
\item[(d)] If $x_1,\dotsc,x_n$ are freely independent and identically
distributed with amalgamation over $\mc B$, and have centered
semicircular distributions with respect to $E$, then the sequence is
$O_n^+$-invariant.
\end{longlist}

\item[(2)] Half-liberated case: Suppose that $(x_1,\dotsc,x_n)$
half-commute, and that $\mc B$ is central in $\mc A$.
\begin{longlist}[(b)]
\item[(a)] If $x_1,\dotsc,x_n$ are half-independent and identically
distributed\break given~$\mc B$, then the sequence is $H_n^*$-invariant.
\item[(b)] If $x_1,\dotsc,x_n$ are half-independent and identically
distributed\break given~$\mc B$, and have symmetrized Rayleigh distributions
with respect to~$E$, then the sequence is $O_n^*$-invariant.
\end{longlist}

\item[(3)] Suppose that $\mc B$ and $x_1,\dotsc,x_n$ generate a commutative algebra.
\begin{longlist}[(b)]
\item[(a)] If $x_1,\dotsc,x_n$ are conditionally independent and
identically distributed given~$\mc B$, then the sequence is $S_n$-invariant.
\item[(b)] If $x_1,\dotsc,x_n$ are conditionally independent and
identically distributed given~$\mc B$, and have even distributions with
respect to $E$, then the sequence is $H_n$-invariant.
\item[(c)] If $x_1,\dotsc,x_n$ are conditionally independent and
identically distributed given~$\mc B$, and have Gaussian distributions
with respect to $E$, then the sequence is $B_n$-invariant.
\item[(d)] If $x_1,\dotsc,x_n$ are conditionally independent and
identically distributed given~$\mc B$, and have centered Gaussian
distributions with respect to $E$, then the sequence is $O_n$-invariant.
\end{longlist}
\end{enumerate}
\end{proposition}

\begin{pf}
Suppose that the joint distribution of $(x_1,\dotsc,x_n)$ satisfies
one of the conditions specified in the statement of the proposition,
and let $D$ be the partition family associated to the corresponding
easy quantum group. By Propositions~\ref{cumulantcharacterization} and
\ref{halfcumulantcharacterization}, and the moment-cumulant formulae,
for any $k \in\N$ and $1 \leq j_1,\dotsc,j_k \leq n$ we have
\begin{eqnarray*}
 &&\sum_{1 \leq i_1,\dotsc,i_k \leq n} \varphi(x_{i_1}\dotsb x_{i_k})
u_{i_1j_1}\dotsb u_{i_kj_k}\\
 && \qquad = \sum_{1 \leq i_1,\dotsc,i_k \leq n}
\varphi(E[x_{j_1}\dotsb x_{j_k}])u_{i_1j_1}\dotsb u_{i_kj_k}\\
 && \qquad = \sum_{1 \leq i_1,\dotsc,i_k \leq n} \mathop{\mathop{\sum}_{\pi\in
D(k)}}_{ \pi\leq\ker\mathbf i} \varphi\bigl(\xi^{(\pi)}_E[x_1,\dotsc
,x_1]\bigr) u_{i_1j_1}\dotsb u_{i_kj_k}\\
 && \qquad = \sum_{\pi\in D(k)} \varphi\bigl(\xi^{(\pi)}_E[x_1,\dotsc,x_1]\bigr)
\mathop{\mathop{\sum}_{1 \leq i_1,\dotsc,i_k \leq n }}_{ \pi\leq\ker\mathbf i}
u_{i_1j_1}\dotsb u_{i_kj_k},
\end{eqnarray*}
where $\xi$ denotes the free, half-liberated or classical cumulants in
cases (1), (2) and (3), respectively. It follows from the
considerations in~\cite{bsp}, or by direct computation, that if $\pi
\in D(k)$ then
\[
\mathop{\mathop{\sum}_{1 \leq i_1,\dotsc,i_k \leq n }}_{ \pi\leq\ker
\mathbf i} u_{i_1j_1}\dotsb u_{i_kj_k} =
\cases{\displaystyle
1_{C(G_n)},  &\quad $\pi\leq\ker\mathbf j$,\cr\displaystyle 0,  &\quad
otherwise.
}
\]
Applying this above, we find
\begin{eqnarray*}
\sum_{  1 \leq i_1,\dotsc,i_k \leq n} \varphi
(x_{i_1}\dotsb x_{i_k})u_{i_1j_1}\dotsb u_{i_kj_k} &=&
\mathop{\mathop{\sum}_{\pi\in D(k)}}_{ \pi\leq\ker\mathbf j} \varphi\bigl(\xi_E^{(\pi
)}[x_1,\dotsc,x_1]\bigr)1_{C(G_n)}\\
&=& \varphi(x_{j_1}\dotsb x_{j_k})1_{C(G_n)},
\end{eqnarray*}
which completes the proof.
\end{pf}

\begin{remark}
To prove the approximation result for finite sequences, we will require
more analytic structure. Throughout the rest of the section, we will
assume that $G_n \subset O_n^+$ is a compact quantum group, $(M,\varphi
)$ is a W$^*$-probability space and $(x_1,\dotsc,x_n)$ is a sequence
of self-adjoint random variables in $M$. We denote the von Neumann
algebra generated by $(x_1,\dotsc,x_n)$ by $M_n$, and define the
\textit{$G_n$-invariant subalgebra} by
\[
\mc B_n = \mathrm{W}^* \bigl(\{p(x)\dvtx  p \in\ms P_n^{\alpha_n}\} \bigr),
\]
where $\ms P_n^{\alpha_n}$ denotes the fixed point algebra of the
action $\alpha_n$, that is,
\[
\ms P_n^{\alpha_n} =\bigl \{p \in\ms P_n\dvtx  \alpha_n(p) = p \otimes
1_{C(G_n)}\bigr\}.
\]
\end{remark}

We now begin the technical preparations for our approximation result.
First, we will need to extend the action $\alpha_n$ to the von Neumann
algebra context. $L^\infty(G_n)$ will denote the von Neumann algebra
obtained by taking the weak closure of $\pi_n(C(G_n))$, where $\pi_n$
is the GNS representation of $C(G_n)$ on the GNS Hilbert space
$L^2(G_n)$ for the Haar state. $L^\infty(G_n)$ is a \textit{Hopf von
Neumann algebra}, with the natural structure induced from $C(G_n)$.
\begin{proposition}\label{coaction}
Suppose that $(x_1,\dotsc,x_n)$ is $G_n$-invariant. Then there is a
right coaction $\widetilde\alpha_n\dvtx  M_n \to M_n \otimes L^\infty
(G_n)$ determined by
\[
\widetilde\alpha_n (p(x)) = (\mathrm{ev}_x \otimes\pi_n) \alpha_n(p)
\]
for $p \in\ms P_n$. Moreover, the fixed point algebra of $\widetilde
\alpha_n$ is precisely the $G_n$-invariant subalgebra $\mc B_n$.
\end{proposition}

\begin{pf}
This follows from~\cite{cur1}, Theorem 3.3, after identifying the GNS
representation of $\ms P_n$ for the state $\varphi_x$ with the
homomorphism \mbox{$\mathrm{ev}_x\dvtx \ms P_n \to M_n$}.
\end{pf}

There is a natural conditional expectation $E_n\dvtx M_n \to\mc B_n$ given
by integrating the coaction $\widetilde\alpha_n$ with respect to the
Haar state, that is,
\[
E_n[m] = \biggl(\mathrm{id} \otimes\int\biggr) \widetilde\alpha_n(m).
\]
By using the Weingarten calculus, we can give a simple combinatorial
formula for the moment functionals with\vadjust{\goodbreak} respect to $E_n$ if $G_n$ is
one of the easy quantum groups under consideration. In the
half-liberated case, we must first show that $\mc B_n$ is central in
$M_n$.\vspace*{-2pt}

\begin{lemma}
Suppose that $(x_1,\dotsc,x_n)$ half-commute. If $H_n^* \subset G_n$,
then the $G_n$-invariant subalgebra $\mc B_n$ is contained in the
center of $M_n$.\vspace*{-2pt}
\end{lemma}

\begin{pf}
Since the $G_n$-invariant subalgebra is clearly contained in the
$H_n^*$-invariant subalgebra, it suffices to prove the result for $G_n
= H_n^*$. Observe that the representation of $G_n$ on the subspace of
$\ms P_n$ consisting of homogeneous noncommutative polynomials of
degree $k$, given by the restriction of $\alpha_n$, is naturally
identified with $u^{\otimes k}$, where $u$ is the fundamental
representation of $G_n$. As discussed in Section~\ref{sec:background},
$\mathrm{Fix}(u^{\otimes k})$ is spanned by the operators $T_\pi$ for
$\pi\in E_h(k)$. It follows that the fixed point algebra of $\alpha
_n$ is spanned by
\[
p_{\pi} = \mathop{\mathop{\sum}_{1 \leq i_1,\dotsc,i_k \leq n}}_{ \pi\leq
\ker\mathbf i} t_{i_1}\dotsb t_{i_k}\vspace*{-1pt}
\]
for $k \in\N$ and $\pi\in E_h(k)$. Therefore, $\mc B_n$ is generated
by $p_\pi(x)$, for $k \in\N$ and $\pi\in E_h(k)$. Recall from
Section~\ref{sec:halfindependence} that if $\omega\in S_k$ is a
parity preserving permutation, then $x_{i_1}\dotsb x_{i_k} =
x_{i_{\omega(1)}}\dotsb x_{i_{\omega(k)}}$ for any $1 \leq i_1,\dotsc
,i_k \leq n$. It follows that $p_{\pi}(x) = p_{\omega(\pi)}(x)$,
where $\omega(\pi)$ is given by the usual action of permutations on
set partitions. Now if $\pi\in E_h(k)$, it is easy to see that there
is a parity preserving permutation $\omega\in S_k$ such that
\[
\omega(\pi) =  \bigl\{(1,\dotsc,2k_1),\dotsc, \bigl(2(k_1 + \dotsb+
k_{l-1})+1,\dotsc,2(k_1+\dotsc+ k_l)\bigr) \bigr\}\vspace*{-1pt}
\]
is an interval partition. We then have
\[
p_{\pi}(x) = p_{\omega(\pi)}(x) =  \Biggl(\sum_{i_1=1}^n
x_{i_1}^{2k_1} \Biggr)\dotsb \Biggl(\sum_{i_l=1}^n
x_{i_l}^{2k_l} \Biggr).
\]
Since $x_i^2$ is central in $M_n$ for $1 \leq i \leq n$, the result follows.\vspace*{-2pt}
\end{pf}

\begin{proposition}\label{cumulants}
Suppose that $(x_1,\dotsc,x_n)$ is $G_n$-invariant, and that one of
the following conditions is satisfied:\vspace*{-1pt}
\begin{longlist}[(3)]
\item[(1)]$G_n$ is a free quantum group $O_n^+,S_n^+,H_n^+$ or $B_n^+$.
\item[(2)]$G_n$ is a half-liberated quantum group $O_n^*$ or $H_n^*$ and
$(x_1,\dotsc,x_n)$ half-commute.
\item[(3)]$G_n$ is an easy group $O_n,S_n,H_n$ or $B_n$ and $(x_1,\dotsc
,x_n)$ commute.\vspace*{-1pt}
\end{longlist}
Then for any $\pi$ in the partition category $D(k)$ for the easy
quantum group~$G_n$, and any $b_0,\dotsc,b_k \in\mc B_n$, we have
\[
E_n^{(\pi)}[b_0x_1b_1,\dotsc,x_1b_k] = \frac{1}{n^{|\pi|}}
\mathop{\mathop{\sum}_{1 \leq i_1,\dotsc,i_k \leq n}}_{ \pi\leq\ker\mathbf i}
b_0x_{i_1}\dotsb x_{i_k}b_k,\vspace*{-1pt}
\]
which holds if $n$ is sufficiently large that the Gram matrix $G_{kn}$
is invertible.\vadjust{\goodbreak}
\end{proposition}

\begin{pf}
We prove this by induction on the number of blocks of $\pi$. First,
suppose that $\pi= 1_k$ is the partition with only one block. Then
\begin{eqnarray*}
E_n^{(1_k)}[b_0x_1b_1,\dotsc,x_1b_k] &=& E_n[b_0x_1\dotsb x_1b_k]\\[-3pt]
&=& \sum_{1 \leq i_1,\dotsc,i_k \leq n } b_0x_{i_1}\dotsb x_{i_k}b_k
\int u_{i_11}\dotsb u_{i_k1},\vspace*{-1pt}
\end{eqnarray*}
where we have used the fact that $b_0,\dotsc,b_k$ are fixed by the
coaction $\widetilde\alpha_n$. Applying the Weingarten integration
formula in Proposition~\ref{weingarten}, we have
\begin{eqnarray*}
E_n[b_0x_1\dotsb x_1b_k] &=& \sum_{1 \leq i_1,\dotsc,i_k \leq n}
b_0x_{i_1}\dotsb x_{i_k}b_k \mathop{\mathop{\sum}_{\sigma,\pi\in D(k)}}_{
\pi\leq\ker\mathbf i} W_{kn}(\pi,\sigma)\\[-3pt]
&=& \sum_{\pi\in D(k)}  \biggl(\sum_{\sigma\in D(k)} W_{kn}(\pi
,\sigma) \biggr) \mathop{\mathop{\sum}_{1 \leq i_1,\dotsc,i_k \leq n }}_{
\pi\leq\ker\mathbf i} b_0x_{i_1}\dotsb x_{i_k}b_k.\vspace*{-1pt}
\end{eqnarray*}
Observe that $G_{kn}(\sigma,1_k) = n^{|\sigma\vee1_k|} = n$ for any
$\sigma\in D(k)$. It follows that for any $\pi\in D(k)$, we have
\begin{eqnarray*}
n \cdot\sum_{\sigma\in D(k)} W_{kn}(\pi,\sigma) &=& \sum_{\sigma
\in D(k)} W_{kn}(\pi,\sigma)G_{kn}(\sigma,1_k)\\[-3pt]
&=& \delta_{\pi1_k}.\vspace*{-1pt}
\end{eqnarray*}
Applying this above, we find
\begin{eqnarray*}
E_n[b_0x_1\dotsb x_1b_k] &=& \sum_{\pi\in D(k)} n^{-1}\delta_{\pi
1_k} \mathop{\mathop{\sum}_{1 \leq i_1,\dotsc,i_k \leq n}}_{ \pi\leq\ker
\mathbf i} b_0x_{i_1}\dotsb x_{i_k}b_k\\[-3pt]
&=& \frac{1}{n}\sum_{i=1}^n b_0x_i\dotsb x_ib_k,\vspace*{-1pt}
\end{eqnarray*}
as desired.

If condition (2) or (3) are satisfied, then the general case follows
from the formula
\[
E_n^{(\pi)}[b_0x_1b_1,\dotsc,x_1b_k] = b_1\dotsb b_k \prod_{V \in
\pi} E_n(V)[x_1,\dotsc,x_1],\vspace*{-1pt}
\]
where in the half-liberated case we are applying the previous lemma.
The one thing we must check here is that if $\pi\in D(k)$ and $V$ is a
block of $\pi$ with $s$ elements, then $1_s \in D(s)$. This is easily
verified, in each case, for $D = P,P_2,P_h,P_b,E_h,E_2$.

Suppose now that condition (1) is satisfied. Let $\pi\in D(k)$. Since
$\pi$ is noncrossing, $\pi$ contains an interval $V = \{l+1,\dotsc
,l+s+1\}$. We then have
\begin{eqnarray*}
&&E_n^{(\pi)}[b_0x_1b_1,\dotsc,x_1b_k]\\[-3pt]
&& \qquad  = E_n^{(\pi\setminus
V)}[b_0x_1b_1,\dotsc,E_n[x_1b_{l+1}\dotsb x_1b_{l+s}]x_1,\dotsc,x_1b_k].\vspace*{-1pt}\vadjust{\goodbreak}
\end{eqnarray*}
To apply induction, we must check that $\pi\setminus V \in D(k-s)$ and
$1_s \in D(s)$. Indeed, this is easily verified for $\mathit{NC},\mathit{NC}_2,\mathit{NC}_h$ and
$\mathit{NC}_b$. Applying induction, we have
\begin{eqnarray*}
&&E_n^{(\pi)}[b_0x_1b_1,\dotsc,x_1b_k]\\[-3pt]
&& \qquad = \frac{1}{n^{|\pi|-1}}\mathop{\mathop{\mathop{\sum}_
{1 \leq i_1,\dotsc,i_l,}}_{ i_{l+s+1},\dotsc,i_k \leq n}}_{
(\pi\setminus V) \leq\ker\mathbf i} b_0x_{i_1}\dotsb b_l
(E_n[x_1b_{l+1}\dotsb x_1b_{l+s}] )x_{i_{l+s}}\dotsb x_{i_k}b_k\\[-3pt]
&& \qquad = \frac{1}{n^{|\pi|-1}}\mathop{\mathop{\mathop{\sum}_{1 \leq
i_1,\dotsc,i_l,}}_{
i_{l+s+1},\dotsc,i_k \leq n}}_{ (\pi\setminus V) \leq\ker\mathbf i}
b_0x_{i_1}\dotsb b_l  \Biggl(\frac{1}{n}\sum_{i=1}^n x_ib_{l+1}\dotsb
x_{i}b_{l+s} \Biggr)x_{i_{l+s}}\dotsb x_{i_k}b_k\\[-3pt]
&& \qquad = \frac{1}{n^{|\pi|}}\mathop{\mathop{\sum}_{1 \leq i_1,\dotsc, i_k \leq
n}}_{ \pi\leq\ker\mathbf i} b_0x_{i_1}\dotsb x_{i_k}b_k,
\end{eqnarray*}
which completes the proof.\vspace*{-2pt}
\end{pf}

We are now prepared to prove the approximation result for finite sequences.\vspace*{-2pt}

\begin{theorem}\label{approx}
Suppose that $(x_1,\dotsc,x_n)$ is $G_n$-invariant, and that one of
the following conditions is satisfied:\vspace*{-1pt}
\begin{longlist}[(3)]
\item[(1)]$G_n$ is a free quantum group $O_n^+,S_n^+,H_n^+$ or $B_n^+$.
\item[(2)]$G_n$ is a half-liberated quantum group $O_n^*$ or $H_n^*$ and
$(x_1,\dotsc,x_n)$ half-commute.
\item[(3)]$G_n$ is an easy group $O_n,S_n,H_n$ or $B_n$ and $(x_1,\dotsc
,x_n)$ commute.\vspace*{-1pt}
\end{longlist}
Let $(y_1,\dotsc,y_n)$ be a sequence of $\mc B_n$-valued random
variables with $\mc B_n$-valued joint distribution determined as follows:\vspace*{-1pt}
\begin{itemize}
\item$G = O^+$: Free semicircular, centered with same variance as $x_1$.
\item$G = S^+$: Freely independent, $y_i$ has same distribution as $x_1$.
\item$G = H^+$: Freely independent, $y_i$ has same distribution as $x_1$.
\item$G = B^+$: Free semicircular, same mean and variance as $x_1$.
\item$G = O^*$: Half-liberated Gaussian, same variance as $x_1$.
\item$G = H^*$: Half-independent, $y_i$ has same distribution as $x_1$.
\item$G = O$: Independent Gaussian, centered with same variance as $x_1$.
\item$G = S$: Independent, $y_i$ has same distribution as $x_1$.
\item$G = H$: Independent, $y_i$ has same distribution as $x_1$.
\item$G = B$: Independent Gaussian, same mean and variance as $x_1$.\vspace*{-1pt}
\end{itemize}
If $1 \leq j_1,\dotsc,j_k \leq n$ and $b_0,\dotsc,b_k \in\mc B_n$, then
\[
 \|E_n[b_0x_{j_1}\dotsb x_{j_k}b_k] - E[b_0y_{j_1}\dotsb
y_{j_k}b_k] \| \leq\frac{C_k(G)}{n}\|x_1\|^k\|b_0\|\dotsb\|b_k\|,\vadjust{\goodbreak}
\]
where $C_k(G)$ is a universal constant which depends only on $k$ and
the easy quantum group $G$.
\end{theorem}

\begin{pf}
First, we note that it suffices to prove the statement for $n$
sufficiently large, in particular we will assume throughout that $n$ is
sufficiently large for the Gram matrix $G_{kn}$ to be invertible.

Let $1 \leq j_1,\dotsc,j_k \leq n$ and $b_0,\dotsc,b_k \in\mc B_n$.
We have
\begin{eqnarray*}
E_n[b_0x_{j_1}\dotsb x_{j_k}b_k] &=& \sum_{1 \leq i_1,\dotsc,i_k \leq
n} b_0x_{i_1}\dotsb x_{i_k}b_k \int u_{i_1j_1}\dotsb u_{i_kj_k}\\
&=& \sum_{1 \leq i_1,\dotsc,i_k \leq n} b_0x_{i_1}\dotsb x_{i_k}b_k
\mathop{\mathop{\mathop{\sum}_{\pi,\sigma\in D(k)}}_{ \pi\leq\ker\mathbf
i}}_{
\sigma\leq\ker\mathbf j} W_{kn}(\pi,\sigma)\\
&=& \mathop{\mathop{\sum}_{\sigma\in D(k)}}_{ \sigma\leq\ker\mathbf j}
\sum_{\pi\in D(k)} W_{kn}(\pi,\sigma) \mathop{\mathop{\sum}_{1 \leq
i_1,\dotsc,i_k \leq n}}_{ \pi\leq\ker\mathbf i} b_0x_{i_1}\dotsb x_{i_k}b_k.
\end{eqnarray*}
On the other hand, it follows from the assumptions on $(y_1,\dotsc
,y_n)$ and the various moment-cumulant formulae that
\[
E[b_0y_{j_1}\dotsb y_{j_k}b_k] = \mathop{\mathop{\sum}_{\sigma\in D(k)}}_{
\sigma\leq\ker\mathbf j} \xi_{E_n}^{(\sigma)}[b_0x_1b_1,\dotsc,x_1b_k],
\]
where $\xi$ denotes the relevant free, classical or half-liberated
cumulants. The right-hand side can be expanded, via M\"{o}bius
inversion, in terms of expectation functionals $E_n^{(\pi
)}[b_0x_1b_1,\dotsc,x_1b_k]$ where $\pi$ is a partition in $\mathit{NC},E_h,P$
according to cases (1), (2), (3), and $\pi\leq\sigma$ for some
$\sigma\in D(k)$. Now if $\pi\notin D(k)$ then we claim that this
expectation functional is zero. Indeed this is only possible if $D =
\mathit{NC}_2,\mathit{NC}_h,P_2,P_h$ and $\pi$ has a block with an odd number of legs.
But it is easy to see that in these cases $x_1$ has an even
distribution with respect to $E_n$, and therefore $E_n^{(\pi
)}[b_0x_1b_1,\dotsc,x_1b_k] = 0$ as claimed. This observation,
together with Proposition~\ref{moebius}, allows to to rewrite the
above equation as
\[
E[b_0y_{j_1}\dotsb y_{j_k}b_k] = \mathop{\mathop{\sum}_{\sigma\in D(k)}}_{
\sigma\leq\ker\mathbf j} \mathop{\mathop{\sum}_{\pi\in D(k)}}_{ \pi\leq
\sigma} \mu_{D(k)}(\pi,\sigma)E_n^{(\pi)}[b_0x_1b_1,\dotsc,x_1b_k].
\]

Applying Lemma~\ref{cumulants}, we have
\[
E[b_0y_{j_1}\dotsb y_{j_k}b_k] = \mathop{\mathop{\sum}_{\sigma\in D(k)}}_{
\sigma\leq\ker\mathbf j} \mathop{\mathop{\sum}_{\pi\in D(k)}}_{ \pi\leq
\sigma} \mu_{D(k)}(\pi,\sigma) n^{-|\pi|}\mathop{\mathop{\sum}_{1 \leq
i_1,\dotsc,i_k \leq n}}_{ \pi\leq\ker\mathbf i} b_0x_{i_1}\dotsb x_{i_k}b_k.
\]
Comparing these two equations, we find that
\begin{eqnarray*}
&&E_n[b_0x_{j_1}\dotsb x_{j_k}b_k] - E[b_0y_{j_1}\dotsb y_{j_k}b_k] \\
&& \qquad = \mathop{\mathop{\sum}_{\sigma\in D(k)}}_{ \sigma\leq\ker\mathbf j}
\sum_{\pi\in D(k)}  \bigl(W_{kn}(\pi,\sigma) - \mu
_{D(k)}(\pi,\sigma)n^{-|\pi|} \bigr)\mathop{\mathop{\sum}_{1 \leq
i_1,\dotsc,i_k \leq n}}_{ \pi\leq\ker\mathbf i} b_0x_{i_1}\dotsb x_{i_k}b_k.
\end{eqnarray*}
Now since $x_1,\dotsc,x_n$ are identically distributed with respect to
the faithful state $\varphi$, it follows that these variables have the
same norm. Therefore,
\[
 \biggl\| \mathop{\mathop{\sum}_{1 \leq i_1,\dotsc,i_k \leq n}}_{ \pi\leq
\ker\mathbf i} b_0x_{i_1}\dotsb x_{i_k}b_k  \biggr\| \leq n^{|\pi|}\|
x_1\|^k\|b_0\|\dotsb\|b_k\|
\]
for any $\pi\in D(k)$. Combining this with former equation, we have
\begin{eqnarray*}
 &&\| E_n[b_0x_{j_1}\dotsb x_{j_k}b_k] - E[b_0y_{j_1}\dotsb
y_{j_k}b_k] \| \\[-2pt]
&& \qquad \leq\mathop{\mathop{\sum}_{\sigma\in D(k)}}_{ \sigma\leq\ker\mathbf j}
\sum_{{\pi\in D(k)}}  \bigl|W_{kn}(\pi,\sigma)n^{|\pi|}
- \mu_{D(k)}(\pi,\sigma)\bigr  | \|x_1\|^k\|b_0\|\dotsb\|b_k\|.
\end{eqnarray*}
Setting
\[
C_k(G) = \sup_{n \in\N}  n \cdot
\sum_{\sigma,\pi\in D(k)}\bigl |W_{kn}(\pi,\sigma
)n^{|\pi|} - \mu_{D(k)}(\pi,\sigma)\bigr|,
\]
which is finite by Proposition~\ref{West}, completes the proof.\vspace*{-2pt}
\end{pf}

\section{Infinite quantum invariant sequences}\label{sec:definetti}

In this section, we will prove Theorem~\ref{definetti}. Throughout
this section, we will assume that $G$ is one of the easy quantum groups
$O,S,H,B,O^*,H^*,O^+,S^+,H^+$ or $B^+$. We will make use of the
inclusions $G_n \hookrightarrow G_{m}$ for $n < m$, which correspond to
the Hopf algebra morphisms $\omega_{n,m}\dvtx C(G_m) \to C(G_n)$ determined by
\[
\omega_{n,m}(u_{ij}) =
\cases{\displaystyle u_{ij},  &\quad $1 \leq i, j \leq n$,\cr\displaystyle \delta_{ij}1_{C(G_n)},  &\quad
$\max\{i,j\} > n$.
}
\]
The existence of $\omega_{n,m}$ may be verified in each case by using
the universal relations of $C(G_n)$.

We begin by extending the notion of $G_n$-invariance to infinite sequences.\vspace*{-2pt}

\begin{definition}
Let $(x_i)_{i \in\N}$ be a sequence in a noncommutative probability
space $(\mc A, \varphi)$. We say that the joint distribution of
$(x_i)_{i \in\N}$ is \textit{invariant under $G$}, or that the
sequence is \textit{$G$-invariant}, if $(x_1,\dotsc,x_n)$ is
$G_n$-invariant for each $n \in\N$.
\end{definition}

This means that the joint distribution functional of $(x_1,\dotsc
,x_n)$ is invariant under the action $\alpha_n\dvtx \ms P_n \to\ms P_n
\otimes C(G_n)$ for each $n \in\N$. It will be convenient to extend
these actions\vadjust{\goodbreak} to $\ms P_\infty= \C\langle t_i\dvtx  i \in\N\rangle$, by
defining $\beta_n\dvtx  \ms P_\infty\to\ms P_\infty\otimes C(G_n)$ to be
the unique unital homomorphism such that
\[
\beta_n(t_j) =
\cases{\displaystyle
\sum_{i=1}^n t_i \otimes u_{ij},  &\quad $1 \leq j \leq n$,\cr\displaystyle
t_j
\otimes1_{C(G_n)},  &\quad $j > n$.
}
\]
It is clear that $\beta_n$ is an action of $G_n$, moreover we have the
relations
\[
(\mathrm{id} \otimes\omega_{n,m}) \circ\beta_{m} = \beta_n
\]
and
\[
(\iota_n \otimes\mathrm{id}) \circ\alpha_n = \beta_n \circ\iota_n,
\]
where $\iota_n\dvtx  \ms P_n \to\ms P_\infty$ is the natural inclusion.
Using these compatibilities, it is not hard to see that a sequence
$(x_i)_{i \in\N}$ is $G$-invariant if and only if the joint
distribution functional $\varphi_x\dvtx \ms P_\infty\to\C$ is invariant
under $\beta_n$ for each $n \in\N$.\looseness=-1

Throughout the rest of the section, $(M,\varphi)$ will be a
W$^*$-probability space and $(x_i)_{i \in\N}$ a sequence of
self-adjoint random variables in $(M,\varphi)$. We will assume that
$M$ is generated as a von Neumann algebra by $\{x_i\dvtx i \in\N\}$.
$L^2(M,\varphi)$ will denote the GNS Hilbert space, with inner product
$\langle m_1, m_2 \rangle= \varphi(m_1^*m_2)$. The strong topology on
$M$ will be taken with respect to the faithful representation on
$L^2(M,\varphi)$. We set
\[
\mc B_n = \mathrm{W}^* \bigl(\{p(x)\dvtx  p \in\ms P_\infty^{\beta_n}\}
 \bigr),
\]
where $\ms P_\infty^{\beta_n}$ is the fixed point algebra of the
action $\beta_n$. Since
\[
(\mathrm{id} \otimes\omega_{n,n+1}) \circ\beta_{n+1} = \beta_n,
\]
it follows that $\mc B_{n+1} \subset\mc B_n$ for all $n \geq1$. We
then define the \textit{$G$-invariant subalgebra} by
\[
\mc B = \bigcap_{n \geq1} \mc B_n.
\]

\begin{remark}
If $(x_i)_{i \in\N}$ is $G$-invariant, then as in Proposition \ref
{coaction}, for each $n \in\N$ there is a right coaction $\widetilde
\beta_n\dvtx M \to M \otimes L^\infty(G_n)$ determined by
\[
\widetilde\beta_n(p(x)) = (\mathrm{ev}_x \otimes\pi_n)\beta_n(p)
\]
for $p \in\ms P_\infty$, and moreover the fixed point algebra of
$\widetilde\beta_n$ is $\mc B_n$. For each \mbox{$n \in\N$}, there is then
a $\varphi$-preserving conditional expectation $E_n\dvtx M \to\mc B_n$
given by integrating the action $\widetilde\beta_n$, that is,
\[
E_n[m] = \biggl(\mathrm{id} \otimes\int\biggr) \widetilde\beta_n(m)
\]
for $m \in M$. By taking the limit as $n\to\infty$, we obtain a
$\varphi$-preserving conditional expectation onto the $G$-invariant subalgebra.
\end{remark}

\begin{proposition}\label{revmart}
Suppose that $(x_i)_{i \in\N}$ is $G$-invariant. Then:\vadjust{\goodbreak}
\begin{longlist}[(3)]
\item[(1)] For any $m \in M$, the sequence $E_n[m]$ converges in $|\cdot|_2$
and the strong topology to a limit $E[m]$ in $\mc B$. Moreover, $E$ is
a $\varphi$-preserving conditional expectation of $M$ onto $\mc B$.
\item[(2)] Fix $\pi\in \mathit{NC}(k)$ and $m_1,\dotsc,m_k \in M$, then
\[
E^{(\pi)}[m_1 \otimes\dotsb\otimes m_k] = \lim_{n \to\infty}
E_{n}^{(\pi)}[m_1 \otimes\dotsb\otimes m_k],\vspace*{-1pt}
\]
with convergence in the strong topology.\vspace*{-3pt}
\end{longlist}
\end{proposition}

\begin{pf}
The proof follows from~\cite{cur2}, Proposition 4.7. Note that (1) is just
a simple noncommutative reversed martingale convergence theorem. More
sophisticated convergence theorems for noncommutative martingales have
been obtained; see, for example,~\cite{gold1,gold2}.\vspace*{-3pt}
\end{pf}

We are now prepared to prove Theorem~\ref{definetti}.\vspace*{-3pt}

\begin{pf*}{Proof of Theorem~\ref{definetti}}
Let $j_1,\dotsc,j_k \in\N$ and $b_0,\dotsc,b_k \in B$. As in the
proof of Theorem~\ref{approx}, we have
\begin{eqnarray*}
&&E[b_0x_{j_1}\dotsb x_{j_k}b_k]\\[-3pt]
 && \qquad = \lim_{n \to\infty}
E_n[b_0x_{j_1}\dotsb x_{j_k}b_k]\\[-3pt]
 && \qquad = \lim_{n \to\infty} \mathop{\mathop{\sum}_{\sigma\in D(k)}}_{ \sigma
\leq\ker\mathbf j} \sum_{\pi\in D(k)} W_{kn}(\pi,\sigma) \mathop{\mathop{\sum}_{1 \leq i_1,\dotsc,i_k \leq n}}_{ \pi\leq\ker\mathbf i}
b_0x_{i_1}\dotsb x_{i_k}b_k\\[-3pt]
 && \qquad = \lim_{n \to\infty} \mathop{\mathop{\sum}_{\sigma\in D(k)}}_{ \sigma
\leq\ker\mathbf j} \mathop{\mathop{\sum}_{\pi\in D(k)}}_{ \pi\leq\sigma
} \mu_{D(k)}(\pi,\sigma) n^{-|\pi|}\mathop{\mathop{\sum}_{1 \leq
i_1,\dotsc,i_k \leq n}}_{ \pi\leq\ker\mathbf i} b_0x_{i_1}\dotsb x_{i_k}b_k.\vspace*{-1pt}
\end{eqnarray*}
By Proposition~\ref{cumulants}, and using the compatibility
\[
(\,\widetilde\iota_n \otimes\mathrm{id}) \circ\widetilde\alpha_n =
\widetilde\beta_n \circ\widetilde\iota_n,\vspace*{-1pt}
\]
where $\widetilde\iota_n\dvtx W^*(x_1,\dotsc,x_n) \to M$ is the obvious
inclusion and $\widetilde\alpha_n$ is as in the previous section, we have
\[
E[b_0x_{j_1}\dotsb x_{j_k}b_k] = \lim_{n \to\infty} \mathop{\mathop{\sum}_{
\sigma\in D(k)}}_{ \sigma\leq\ker\mathbf j} \mathop{\mathop{\sum}_{\pi
\in D(k)}}_{ \pi\leq\sigma} \mu_{D(k)}(\pi,\sigma) E_n^{(\pi
)}[b_0x_1b_1,\dotsc,x_1b_k].\vspace*{-1pt}
\]
By (2) of Proposition~\ref{revmart}, we obtain
\[
E[b_0x_{j_1}\dotsb x_{j_k}b_k] = \mathop{\mathop{\sum}_{\sigma\in D(k)}}_{
\sigma\leq\ker\mathbf j} \mathop{\mathop{\sum}_{\pi\in D(k)}}_{ \pi\leq
\sigma} \mu_{D(k)}(\pi,\sigma) E^{(\pi)}[b_0x_1b_1,\dotsc,x_1b_k].\vspace*{-1pt}
\]
As discussed in the proof of Theorem~\ref{approx}, we can replace the
sum of expectation functionals by cumulants to obtain
\[
E[b_0x_{j_1}\dotsb x_{j_k}b_k] = \mathop{\mathop{\sum}_{\sigma\in D(k)}}_{
\sigma\leq\ker\mathbf j} \xi_E^{(\sigma)}[b_0x_1b_1,\dotsc,x_1b_k],\vspace*{-1pt}\vadjust{\goodbreak}
\]
where $\xi$ denotes the relevant free, half-liberated or classical
cumulants. Since the cumulants are determined by the moment-cumulant
formulae, we find that
\[
\xi_E^{(\sigma)}[b_0x_{j_1}b_1,\dotsc,x_{j_k}b_k] =
\cases{\displaystyle
 \xi_E^{(\sigma)}[b_0x_1b_1,\dotsc,x_1b_k],  &\quad$\sigma\in
D(k)$   and   $\sigma\leq\ker\mathbf j$,\cr\displaystyle 0,  &\quad
otherwise.
}
\]
The result then follows from the characterizations of these joint
distributions in terms of cumulants given in Theorem \ref
{cumulantcharacterization} and Propositions \ref
{halfcumulantcharacterization} and~\ref{halfgausscharacterization}.
\end{pf*}

\begin{remark}
For simplicity, we have restricted to elements of a von Neumann
algebra, that is, bounded random variables, in the statement of Theorem
\ref{definetti}. However, for the easy quantum groups $O,B$ and $O^*$
the result implies that the variables must have unbounded
distributions. In the classical setting, the boundedness assumption can
be easily replaced by the condition that $x_1$ has finite moments of
all orders. The key differences are as follows:

First, in the classical case one can replace the uniform bound in
Theorem~\ref{approx} by the $L^p$ estimate
\[
 |E_n[x_{j_1}\dotsb x_{j_k}] - E[y_{j_1}\dotsb y_{j_k}] |_p
\leq\frac{C_k(G)}{n}|x_1|^k_{pk},
\]
where $|\cdot|_p$ denotes the $L^p$-norm. The proof is identical, except
that one uses H\"{o}lder's identity $|x_{i_1}\dotsb x_{i_k}|_p \leq
|x_1|_{pk}^k$ for any $1 \leq i_1,\dotsc,i_p \leq n$.

Second, Proposition~\ref{revmart} is replaced by a standard $L^p$
reversed martingale convergence theorem (the statement for expectation
functionals requiring another application of H\"{o}lder).

With these technical modifications, the proof of Theorem \ref
{definetti} shows that any infinite $B$ (resp., $O$) invariant sequence
of classical random variables with finite moments of all orders has the
same joint moments with respect to~$\mc B$ as a conditionally i.i.d.
(centered) Gaussian family. But this is sufficient to determine the
joint distribution with respect to $\mc B$, since the Gaussian
distribution is characterized by its moments.

Likewise, the result for $O^*$ still holds if $(x_i)_{i \in\N}$ are
of the form in Example~\ref{halfexample}, where $|\xi_i|$ has finite
moments of all orders. The details are left to the reader.
\end{remark}

\section{Concluding remarks}\label{sec:conclusion}

We have seen in this paper that the ``easiness'' condition from \cite
{bsp} provides a good framework for the study of de Finetti type
theorems for orthogonal quantum groups.

A first natural question is what happens in the unitary case. For the
classical unitary group $U_n$, it is well known that an infinite
sequence of complex-valued random variables is unitarily invariant if
and only if they are conditionally i.i.d. centered complex Gaussians.
For the free unitary group~$U_n^+$ this is considered in~\cite{cur2},
where it is shown that an infinite sequence of noncommutative random
variables is quantum unitarily invariant if and only if they form an
operator-valued free circular family with mean zero and common
variance. However, the study and classification of easy quantum groups
seems to be a quite difficult combinatorial problem in the unitary
case, we refer to the concluding section of~\cite{bsp} for a
discussion here.

In addition to the 14 easy quantum groups discussed in this paper,
there are also two infinite series $H_n^{(s)}$ and $H_n^{[s]}$, $s =
2,3,\dotsc,\infty$, which are related to the complex reflection
groups $H_n^s = \Z_s \wr S_n$. These are described in~\cite{bcs1},
with the conjectural conclusion that the class of easy quantum groups
consists of the 14 examples discussed in this paper, and a
multi-parameter ``hyperoctahedral series'' unifying $H_n^{(s)}$ and
$H_n^{[s]}$. It is a natural question whether there are de Finetti type
results for this series, with corresponding notions of
``independence,'' and we plan to return to this question after
completing the construction.

A third question is whether the approximation result in Theorem \ref
{approx} can be strengthened. The main tool that we have available at
this time, namely the Weingarten formula, is only suitable for
estimates on the joint moments. In~\cite{df1}, Diaconis and Freedman
give refined estimates on the variation norm between the distribution
of the coordinates $(u_{11},\dotsc,u_{1k})$ on $S_n$ (resp., $O_n$) and
an independent Bernoulli (resp., Gaussian) distribution. This is used to
prove finite de Finetti type results, where the approximations hold in
variation norm. It is known from~\cite{bc1,bc2} that the coordinates
$(u_{11},\dotsc,u_{1k})$ on~$S_n^+$ and $O_n^+$ converge in moments to
freely independent Bernoulli and semicircular distributions, and it is
a natural question whether these converge in a stronger sense. For
$k=1$, it is known from~\cite{bcz} that the distribution of
$n^{1/2}u_{11}$ in $C(O_n^+)$ ``superconverges'' (in the sense of \cite
{bev}) to the semicircle law, but nothing is currently known for $k > 1$.

Another question is whether the results of Aldous~\cite{ald} for
invariant arrays of random variables have suitable extensions to easy
quantum groups. We will consider this problem first for free quantum
groups in a forthcoming paper~\cite{cs2}.

Another basic symmetry for a sequence of classical random variables is
\textit{spreadability}, that is, invariance under taking subsequences.
Ryll-Nardzewski proved in~\cite{rn} that de Finetti's theorem in fact
holds under this apparently weaker condition. A free analogue of this
condition, and of Ryll-Nardzewski's theorem, has been obtained in~\cite{cur3}.

Finally, there is the general question of applying our ``$S_n,O_n$
philosophy'' to other situations. In~\cite{bcs2}, we have developed a
global approach, using the ``easiness'' formalism, to the fundamental
stochastic eigenvalue computations of Diaconis and Shahshahani~\cite{dsh}.

%

\printaddresses

\end{document}